\documentclass[10pt, a4paper]{article}

\usepackage[arxiv]{optional}

\date{October 28, 2013}

\usepackage{fancyhdr}
\usepackage{graphicx}
\usepackage{amssymb}
\usepackage{amsmath}
\usepackage{paralist}
\usepackage{subfigure}

\opt{arxiv}{%
\usepackage{amsthm}
\newtheorem{theorem}{Theorem}[section]
\newtheorem{lemma}[theorem]{Lemma}

\newtheorem{proposition}[theorem]{Proposition}
\newtheorem{definition}[theorem]{Definition}
}%

\newcommand{\nset}[1]{\ensuremath{\{1,\dotsc,#1\}}}

\newcommand{\R}{\ensuremath{\mathbb R}}
\newcommand{\N}{\ensuremath{\mathbb N}}
\newcommand{\Z}{\ensuremath{\mathbb Z}}

\DeclareMathOperator*{\argmax}{arg\,max}

\newcommand{\card}[1]{\ensuremath{| {#1} |}}

\newcommand{\nnz}[1]{\ensuremath{\mbox{nnz}(#1)}}

\newcommand{\setcond}{\ensuremath{\mid}}

\newcommand{\bigo}[1]{\ensuremath{\mathcal{O}\left(#1\right)}}

\newcommand{\defby}{\mathrel{\mathop:}=}

\newcommand{\bydef}{=\mathrel{\mathop:}}

\newcommand{\flopop}[1] {\ensuremath{\text{flop}(#1)}}

\newcommand{\qcc}{\ensuremath{\text{qcc}}}

\newcommand{\quadcost}{\ensuremath{\text{q}}}

\newcommand{\lincost}{\ensuremath{\text{l}}}

\newcommand{\nhd}[2][{}]{\ensuremath{\mathcal{N}_{#1}\left(#2\right)}}

\newcommand{\cnhd}[2][{}]{\ensuremath{\mathcal{N}_{#1}\left[#2\right]}}

\newcommand{\vdeg}[2][{}]{\ensuremath{\mbox{d}_{#1}(#2)}}

\newcommand{\cvdeg}[2][{}]{\ensuremath{\mbox{d}_{#1}[#2]}}

\newcommand{\ecut}[1]{\ensuremath{\delta({#1})}}

\newcommand{\probname}[1]{{\textsc{#1}}}

\graphicspath{{figures/}}

\title{On the minimum FLOPs problem in the sparse Cholesky
factorization}
\author{Robert Luce\thanks{%
TU Berlin, Institut f\"ur Mathematik, MA 3-3, Stra{\ss}e des 17. Juni
136, 10623 Berlin, Germany, luce@math.tu-berlin.de}
\and Esmond Ng\thanks{%
Lawrence Berkeley National Laboratory, One Cyclotron Road, Mail Stop
50F-1650, Berkeley, CA 94720-8139, USA, EGNg@lbl.gov
}}

\begin{document}

\maketitle

\thispagestyle{fancy}
\setlength{\headheight}{15pt}
\chead{\scriptsize\revision}

\begin{abstract} Prior to computing the Cholesky factorization of a
sparse symmetric positive definite matrix, a reordering of the rows
and columns is computed so as to reduce both the number of fill
elements in Cholesky factor and the number of arithmetic operations
(FLOPs) in the numerical factorization. These two metrics are clearly
somehow related and yet it is suspected that these two problems are
different.  However, no rigorous theoretical treatment of the relation
of these two problems seems to have been given yet.  In this paper we
show by means of an explicit, scalable construction that the two
problems are different in a very strict sense: no ordering is optimal
for both fill and FLOPs in the constructed graph.

Further, it is commonly believed that minimizing the number of FLOPs
is no easier than minimizing the fill (in the complexity sense), but
so far no proof appears to be known.  We give a reduction chain that
shows the NP hardness of minimizing the number of arithmetic
operations in the Cholesky factorization.

\end{abstract}

\opt{siam}{%
}%

\opt{siam}{%
}%

\pagestyle{myheadings}
\thispagestyle{plain}
\markboth{R. LUCE AND E. NG}{MINIMUM FLOPS IN THE SPARSE CHOLESKY
FACTORIZATION}

\section{Introduction}

Let $A \in \R^{n \times n}$ be a sparse, real, symmetric positive
definite (SPD) matrix and consider the Cholesky factorization of $A$
with symmetric pivoting, that is, $PAP^T = LL^T$, where $L$ is a lower
triangular matrix and $P$ is a permutation matrix.  Assuming no
accidental cancellation, the nonzero pattern of $L+L^T$ depends solely
on the choice of $P$ and contains the nonzero pattern of $PAP^T$. Nonzero
elements of $L$ at positions that are structural zeros in $PAP^T$ are
called {\itshape fill} elements.  Determining a permutation matrix
$P$, such that the number of these fill elements is minimum, is an NP
hard problem~\cite{Yannakakis:1981}.  Since the arithmetic work in
terms of floating point operations for the computation of the Cholesky
factor $L$ is solely determined by the permutation matrix $P$ as well,
one may wonder how the number of fill elements and arithmetic work are
related.  In this paper we study this relationship and give an NP
hardness result for the minimization of the arithmetic work.

Gaussian elimination for symmetric matrices is very conveniently
described in terms of undirected graphs.  For example, the Cholesky
factorization of $A$ can be seen as an embedding of the graph $G(A)$
of $A$ into a triangulated supergraph $G^+$ of $G$.  In this work we assume
familiarity with some basic graph theoretic terminology and concepts such
as the elimination game, chordality and perfect elimination
orderings (PEOs). Useful references that cover all the terminology
we use are~\cite{Rose:1972} and~\cite{Heggernes:2006}.

Let $G = (V, E)$ be a simple undirected graph with $n$ vertices.  If
$F \subseteq V \times V \setminus E$ is a set of fill edges such that 
$G^{+} = (V, E \cup F)$ is chordal, then there exists a PEO
$\alpha : V \rightarrow \{1,\dotsc,n\}$ for $G^+$.
When carrying out
vertex elimination on $G^{+}$ according to $\alpha$, denote by
$\vdeg{\alpha^{-1}(i)}$ the degree of the $i$-th vertex in the course
of the elimination process (the \emph{elimination degree} of
$\alpha^{-1}(i))$.  Minimizing the quantity
\begin{equation*}
    \nnz{\alpha} = \sum_{i=1}^n (\vdeg{\alpha^{-1}(i)} + 1)
\end{equation*}
over all triangulations $G^+ = (V, E \cup F)$ is what we call the
\probname{MinimumFill} problem in this work (equivalently, one could
minimize $\card{F}$).  If $G$ is the graph of a sparse symmetric
positive definite matrix $A$, then $\nnz{\alpha}$ is the number of
nonzero elements in the Cholesky factor of $A$ when carrying out the
factorization in the ordering $\alpha$.

Another metric of interest is the number of floating point operations
(FLOPs) that are required for the computation of the Cholesky factor
in the given ordering $\alpha$.  If we account for all additive,
multiplicative and square-root operations for the computation of the
Cholesky factor, the total number of such FLOPs is given by
\begin{equation*}
    \flopop{\alpha} = \sum_{i=1}^n (\vdeg{\alpha^{-1}(i)} + 1)^2.
\end{equation*}
Minimizing
$\flopop{\alpha}$ over all triangulations of $G$ is the
\probname{MinimumFLOPs} problem.

It is important to note that the multiset of elimination degrees
$\{\vdeg{\alpha^{-1}(i)}\}_{i=1}^n$ is the same for all PEOs $\alpha$
of a triangulation~\cite[Thm. 4]{Rose:1972}.  Hence, the quantities
$\nnz{\cdot}$ and $\flopop{\cdot}$ depend only on the triangulation
$G^+$ (see also \cite{DuffReid:1983:work}).

The \probname{MinimumFLOPs} problem has received much less attention
in the literature than the \probname{MinimumFill} problem. It is also
occasionally noted that the two metrics are related (e.g. \cite[\S
7]{GeorgeLiu:1989}, \cite[ch. 59]{MehtaSahni:2005}) and it is
occasionally noted that the two problems are believed to be
different~(e.g. \cite[sec. 4.1.2]{Rose:1972}).  However, a rigorous
investigation of the relation of these two problems seems to be
missing in the literature.

In section \ref{sec:fill_vs_flops} we discuss a class of graphs,
parameterized by the number of vertices, for which all optimal
orderings with respect to either one metric are strictly suboptimal
for the other.  A third ordering problem to which we relate these
findings is the \probname{Treewidth} problem.  In the context of
multifrontal methods~\cite{DuffReid:1983:mf, Liu:1992}, this problem
asks for an elimination ordering such that the largest front size is
minimum~\cite{BodlaenderEtal:1995}.  It is also a parameter in the
lower bound for the amount of communication in the parallel sparse
Cholesky factorization, since it determines the size of a largest
dense submatrix that has to be factorized.  Finally, we briefly
discuss ordering heuristics from the viewpoint of the minimum FLOPs
problem.

In section \ref{sec:flops_np_hard} we give a formal NP hardness result
for \probname{MinimumFLOPs}.  While it is well known that minimizing
the fill is NP hard~\cite{Yannakakis:1981} and one expects that
minimizing the number of arithmetic operations is no less difficult,
it seems that such a proof has not been given before.

\subsection{Notation}
\label{sec:notation}

We use the following notation throughout this paper.  The Cartesian
product of two sets $P$ and $Q$ is denoted by $P \times Q$.  For two
graphs $G_1 = (V_1, E_1)$ and $G_2 = (V_2, E_2)$ we define their sum
$G_1 + G_2 \defby (V_1 \cup V_2, E_1 \cup E_2)$ and their join $G_1
\vee G_2 \defby (V_1 \cup V_2, E_1 \cup E_2 \cup (V_1 \times V_2))$.
By $K_s$ we refer to the complete graph (or clique) on $s$ vertices.
For a graph $G=(V,E)$ and a vertex $v \in V$, we denote by $\nhd[G]{v}
\subseteq V$ the neighborhood of $v$ in $G$, that is, the vertices
adjacent to $v$.  The closed neighborhood of $v$ is $\cnhd[G]{v}
\defby \nhd[G]{v} \cup \{v\}$.  Denote the vertex degree and the closed vertex
degree of $v$ by $\vdeg[G]{v} = \card{\nhd[G]{v}}$ and $\cvdeg[G]{v} =
\card{\cnhd[G]{v}}$, respectively.  We omit the reference to the graph
$G$ in the notation whenever the context permits.  For example, in the
context of vertex elimination, $\cvdeg{\alpha^{-1}(i)}$ always refers to the
$i$-th elimination degree.
Sometimes we explicitly refer to the
vertex and edge sets of a graph $G$ by $V(G)$ and $E(G)$.
Using this notation we
formally restate the two problems of interest as decision problems
(recall that $\vdeg{\alpha^{-1}(i)}$ refers to the {\itshape
elimination degree} and notice that $\vdeg{\alpha^{-1}(i)} + 1 =
\cvdeg{\alpha^{-1}(i)}$).

\begin{center}
\framebox{
\begin{minipage}{.8\textwidth}
\probname{MinimumFill}\\
Instance: Graph $G=(V,E), n = \card{V}, k\in \N$\\
Question: Is there a set of edges $F\subseteq V\times V$ such that $(V,
E\cup F)$ has a PEO $\alpha : V \rightarrow \nset{n}$ with
$\sum_{i=1}^n \cvdeg{\alpha^{-1}(i)} \le k$?
\end{minipage}
}
\end{center}

\begin{center}
\framebox{
\begin{minipage}{.8\textwidth}
\probname{MinimumFLOPs}\\
Instance: Graph $G=(V,E), n = \card{V}, k\in \N$\\
Question: Is there a set of edges $F\subseteq V\times V$ such that $(V,
E\cup F)$ has a PEO $\alpha : V \rightarrow \nset{n}$ with
$\sum_{i=1}^n \cvdeg{\alpha^{-1}(i)}^2 \le k$?
\end{minipage}
}
\end{center}

\section{Minimum fill and minimum FLOPs are different}
\label{sec:fill_vs_flops}

In this section we present a class of graphs for which minimizing fill
and minimizing FLOPs are different problems.  Interestingly, a
structurally similar class of graphs is used in \cite[p.
14]{Kloks:1994} to show that \probname{MinimumFill} and
\probname{Treewidth} are different.  The treewidth problem is yet
another NP hard problem \cite{ArnborgEtal:1987} that can be formulated
using elimination degrees:
\begin{center}
\framebox{
\begin{minipage}{.8\textwidth}
\probname{Treewidth}\\ Instance: Graph $G=(V,E), n = \card{V}, k\in
\N$\\ Question: Is there a set of edges $F\subseteq V\times V$ such that
$(V, E\cup F)$ has a PEO $\alpha : V \rightarrow \nset{n}$ with
$\max_i \cvdeg{\alpha^{-1}(i)} \le k$?
\end{minipage}
}
\end{center}
We will use the abbreviation $\omega(\alpha) \defby \max_i
\cvdeg{\alpha^{-1}(i)}$, which is exactly the clique number of
the triangulation of $G$ corresponding to $\alpha$.

We will show that \probname{MinimumFill},
\probname{MinimumFLOPs} and \probname{Treewidth} are different
problems in a very strict sense.  In section \ref{sec:icc_graph} we
explore all minimal triangulations  of a parameterized class of graphs
(again, see~\cite{Heggernes:2006} for an overview of the terminology).
Using specific values for the parameters in section
\ref{sec:minflop_neq_minfill_neq_tw}, we show that minima for the
three optimization problems are attained at distinct triangulations.
Finally, in section \ref{sec:heuristics} we discuss the minimum FLOPs
problem from the viewpoint of ordering heuristics.

\subsection{An instructive class of graphs}
\label{sec:icc_graph}

In this section we study a class of graphs whose set of minimal
triangulations is sufficiently simple to analyse and yet general
enough to show that the extrema of minimum fill and minimum FLOPs are
attained at different triangulations.  In \cite[p.  14]{Kloks:1994} it
is pointed out that \probname{MinimumFill} and \probname{Treewidth}
are different problems using graphs from this class.  In that
monograph the author refers to an unpublished report for the details.
Our study covers this aspect as well.

A useful reference for all facts and results on minimal triangulations
which we assume here is the survey by Heggernes~\cite{Heggernes:2006}.  We recall
that every inclusion minimal triangulation can be obtained through
vertex elimination along some elimination ordering.  Such orderings
are called \emph{minimal elimination orderings} (MEOs).

\begin{figure}[t]
    \begin{center}
        \includegraphics{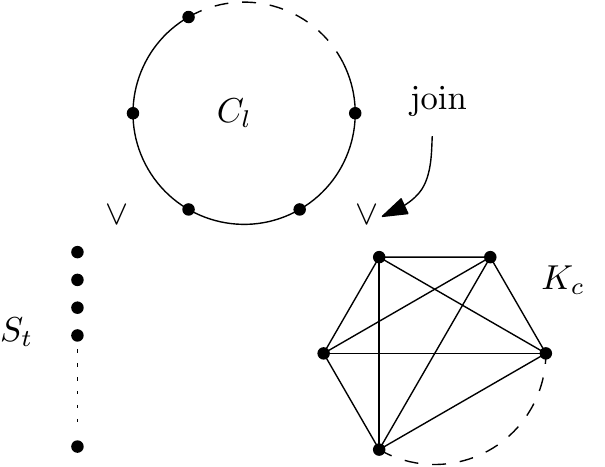}
    \end{center}
    \caption{The graph $G(l,t,c)$\label{fig:gicc}}
\end{figure}

The graph we want to study consists of a cycle $C_l$ on $l$ vertices, a
clique $K_c$ on $c$ vertices and an independent set $S_t$ of $t$ vertices, plus
all possible edges between the cycle and the other $t+c$ vertices (see
Fig. \ref{fig:gicc}).  More formally, for numbers $4 \le l, t, c \in
\N$ the graph is defined as $G = C_l \vee (S_t + K_c)$.
First we will characterize all minimal triangulations of
$G$.  In fact only two types of triangulations exist; they are shown
in Fig.  \ref{fig:gicc_triang}.

\begin{figure}[t]
\includegraphics{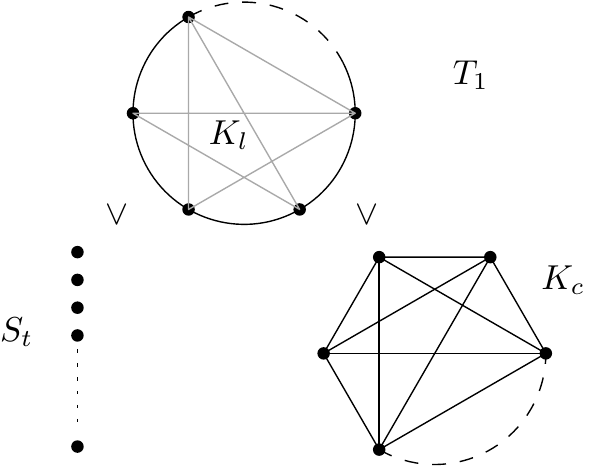}
\hfill
\includegraphics{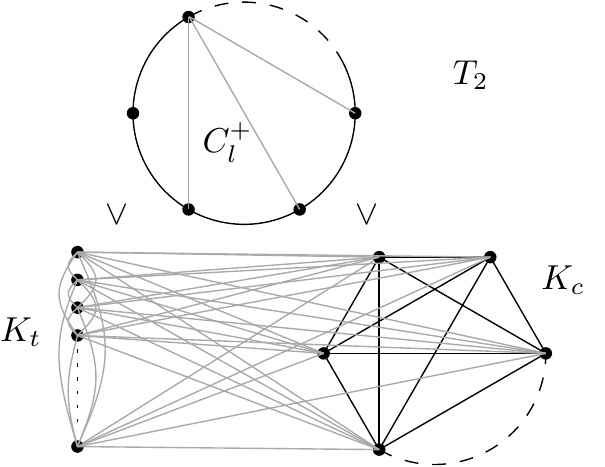}
\caption{The two types of triangulations of $G(l,t,c)$, $T_1$ and
$T_2$. Gray edges are fill edges.\label{fig:gicc_triang}}
\end{figure}

\begin{proposition}
The graph $G \defby C_l \vee (S_t + K_c)$ has exactly two
types of minimal triangulations $T_1 \cong K_l \vee (S_t + 
K_c)$ and $T_2 \cong C_l^+ \vee K_{t+c}$, where $C_l^+$ is a
minimal triangulation of $C_l$.
\end{proposition}
\begin{proof}
It is easy to verify that $T_1$ and $T_2$ are indeed chordal graphs,
since corresponding PEOs are readily constructed.  Let $T$ be a
minimal triangulation of $G$.  Then there exists a minimal
elimination ordering $\alpha : V(G) \rightarrow \nset{l+t+c}$ for $G$ whose
resulting filled graph is $T$.  Let $v = \alpha^{-1}(1)$ be the first
vertex to be eliminated and denote the graph arising from eliminating
$v$ by $G_v^+$. We distinguish three cases:

\begin{asparaenum}
\item[Case 1]: $v \in V(K_c)$.  Then $G^{+}_v \cong K_l \vee (S_t
    + K_{c-1})$, which is a chordal graph.  Since $\alpha$ is a
    MEO for $G$, $\{\alpha^{-1}(2), \dotsc , \alpha^{-1}(n)\}$ is a
    PEO for $G^{+}_v$ and so $T \cong T_1$.

\item[Case 2]: $v \in V(S_t)$.  Then $G_v^+ \cong K_l \vee (S_{t-1}
    + K_c)$, which is a chordal graph.  Since $\alpha$ is a MEO
    for $G$, $\{\alpha^{-1}(2), \dotsc , \alpha^{-1}(n)\}$ is a PEO
    for $G^{+}_v$ and so $T \cong T_1$.

\item[Case 3]: $v \in V(C_l)$.  Then $G^{+}_v \cong C_{l-1} \vee
    K_{t+c}$.  In this graph the only chordless cycle of length at
    least four can possibly be $C_{l-1}$. So the minimal
    triangulations of $G^{+}_v$ are now given by the minimal
    triangulations of $C_{l-1}$, which implies that $T \cong T_2$.
\end{asparaenum}

It remains to show that $T_1$ and $T_2$ are minimal.  We do so by
showing that in both triangulations every fill edge is the unique
chord of some four-cycle in $T_1$ and $T_2$.  For $T_1$ consider any fill
edge $f = (c_i, c_j)$ in $V(C_l) \times V(C_l)$ and $s \in V(S_t), v
\in V(K_c)$.  Then $(s, c_i, v, c_j, s)$ is a four-cycle in $T_1$
whose unique chord is $f$.  For $T_2$ let $f = (s, v)$ be a fill edge
with $s \in V(S_t)$, $v \in V(C_k)$ and $c_1$, $c_2$ two non-adjacent
vertices in $T_2$.  Then $(c_1, s, c_2, v, c_1)$ is a four-cycle in
$T_2$ whose unique chord is $f$.
\end{proof}

For both triangulations, we will now determine the elimination degree
sequence of certain PEOs and count the number of nonzero elements in
the corresponding Cholesky factors as well as the number of FLOPs
necessary to compute them.  

A PEO $\alpha_1$ for $T_1$ is given by ordering the $t$ vertices of
$S_t$ first, followed by any ordering of the remaining complete graph
of size $l+c$.  For the elimination degree sequence we obtain:
\begin{equation*}
    \{\vdeg{\alpha_1^{-1}(i)}\}_{i=1}^{l+t+c}
        = \{l\}_{j=1}^{t} \cup \{l + c - j\}_{j=1}^{l+c}.
\end{equation*}
Given that degree sequence, the nonzero-, FLOP count and clique number
for the Cholesky factor corresponding to $T_1$ are given by
\begin{align}
    \nnz{\alpha_1} & = \sum_{j} (\vdeg{\alpha_1^{-1}(j)} + 1) = t(l+1) +
        \sum_{j=1}^{l+c} j \label{eq:t1_nnz}\\
    \flopop{\alpha_1} & = \sum_{j} (\vdeg{\alpha_1^{-1}(j)} + 1)^2
        = t(l+1)^2 + \sum_{j=1}^{l+c} j^2 \label{eq:t1_flops}\\
    \omega(\alpha_1) & = \max_i {\vdeg{\alpha_1^{-1}(i)} + 1} = l + c
        \label{eq:t1_omega}
\end{align}
Another PEO for $T_1$ is obtained by ordering the vertices of $K_c$
first, followed by the vertices of $S_t$ and finally the vertices of
$K_l$.  Of course, the expressions
\eqref{eq:t1_nnz}--\eqref{eq:t1_omega} are the same for all PEOs. 

A PEO $\alpha_2$ for the triangulation $T_2$ is obtained by the first
$l-2$ vertices of a PEO for $C_l^{+}$ followed by an arbitrary
ordering of the vertices of the remaining $K_{t+c+2}$.  Noting that
for every PEO of $C_l^+$ the elimination degree of the first $l-2$
vertices is $t+c+2$, we obtain the degree sequence
\begin{equation*}
    \{ \vdeg{\alpha_2^{-1}(i)}\}_{i=1}^{l+t+c}
        = \{t+c+2\}_{j=1}^{l-2} \cup \{t+c+2 - j\}_{j=1}^{t+c+2}.
\end{equation*}
The resulting nonzero-, FLOP count and clique number are:
\begin{align}
    \nnz{\alpha_2} & = \sum_{j} (\vdeg{\alpha_2^{-1}(j)} + 1)
        = (l-2)(t+c+3) + \sum_{j=1}^{t+c+2} j \label{eq:t2_nnz}\\
    \flopop{\alpha_2} & = \sum_{j} (\vdeg{\alpha_2^{-1}(j)} + 1)^2
        = (l-2)(t+c+3)^2 + \sum_{j=1}^{t+c+2} j^2 \label{eq:t2_flops}\\
    \omega(\alpha_2) & = \max_i {\vdeg{\alpha_2^{-1}(i)} + 1} = t + c + 3
        \label{eq:t2_omega}
\end{align}

\subsection{Minimizing FLOPs, fill and treewidth are different
problems}
\label{sec:minflop_neq_minfill_neq_tw}

Let $64 < n \in \N$ and set $l=8n$, $t = 5n$, $c=4n$ and consider the
class of graphs from section \ref{sec:icc_graph} with these
parameters.  We will count the number of nonzeros and FLOPs for the
two triangulations.  Using \eqref{eq:t1_nnz}--\eqref{eq:t1_omega} and
\eqref{eq:t2_nnz}--\eqref{eq:t2_omega} we obtain
\begin{align*}
    \nnz{\alpha_1} & = 112n^2 + \bigo{n} &
    \nnz{\alpha_2} & = \frac{225}{2} n^2 + \bigo{n}\\
    \flopop{\alpha_1} & = 896 n^3 + \bigo{n^2} &
    \flopop{\alpha_2} & = 891 n^3 + \bigo{n^2}\\
    \omega(\alpha_1) & = 12n &
    \omega(\alpha_2) & = 9n+3,
\end{align*}
and it is readily verified that the omitted lower order terms are
dominated by the leading terms if $n>64$.  So for this choice of
values for $l$, $t$, $c$, we see that $\alpha_1$ yields the optimal
triangulation for the fill, but not for the number of FLOPs or the
size of the largest clique.  The latter two metrics are minimized by
$\alpha_2$, which is suboptimal for the fill.

If the values $l=2n+3, t=n, c=2n$, $n>3$, are chosen, one obtains the
class of graphs from Kloks' example~\cite[p.14]{Kloks:1994}.  In that
case $\alpha_1$ minimizes both the fill and the number of FLOPs, but not
the size of the largest clique.  The minimum clique size is attained by
$\alpha_2$, which is suboptimal for the fill and FLOPs:
\begin{align*}
    \nnz{\alpha_1} & = 10n^2 + \bigo{n} &
    \nnz{\alpha_2} & = \frac{21}{2} n^2 + \bigo{n}\\
    \flopop{\alpha_1} & = \frac{76}{3} n^3 + \bigo{n^2} &
    \flopop{\alpha_2} & = 27 n^3 + \bigo{n^2}\\
    \omega(\alpha_1) & = 4n+3 &
    \omega(\alpha_2) & = 3n+3
\end{align*}

\begin{theorem}
\label{thm:problems_are_different}
    The three chordal graph embedding problems \probname{MinimumFill},
    \probname{MinimumFLOPs} and \probname{Treewidth}  are different in
    the sense that no two such metrics can be minimized simultaneously
    in general.
\end{theorem}

The three problems above are equivalent to minimizing the $1$-, $2$-
and $\infty$-norm of the vector of elimination degrees over the set of
all chordal embeddings.  It would be interesting to learn whether
\emph{all} such $p$-norm minimization problems for, say, $p \in
[1,\infty]$ are different in the sense of
Thm.~\ref{thm:problems_are_different}.  We did some very preliminary
but encouraging experiments for some pairs of $p$-norms, but did not
pursue this question rigorously.

\subsection{Minimum FLOPs and heuristics}
\label{sec:heuristics}

The minimum degree (MD) heuristic and its variations (e.g.
AMD~\cite{Amestoy:1996}, MMD~\cite{Liu:1985}) are a popular class of
ordering heuristics commonly used to reduce the number of fill
elements in the Cholesky factor.  These heuristics use the elimination
degree of the vertices as their primary local criterion for ordering
the vertices.  Note that this criterion is in fact \emph{the}
canonical local criterion for minimizing the FLOPs and not the fill,
in which context MD type heuristics are usually put.

The canonical criterion for locally minimizing the number of fill
elements is the \emph{deficiency} of a vertex, which accounts for the
number of fill edges the elimination of the vertex would imply.  It
has been observed~\cite{NgRaghavan:1999, RothbergEisenstat:1998} that
using this criterion (or approximations of it) instead of the
elimination degree usually results in less arithmetic (and fill).  In
fact, the authors of~\cite{RothbergEisenstat:1998} regard reducing the
number of FLOPs as their primary objective for their experiments with
the deficiency criterion.

Reported experimental results for ordering heuristics like the ones
above certainly have contributed to the common understanding that
reducing the number of fill elements usually goes hand in hand with
reducing the number of arithmetic operations and vice versa.  While
this behaviour is \emph{typically} observed when ordering heuristics
are benchmarked, it is worth pointing out that it may actually happen
in practice that an ordering that implies less fill than another
ordering actually causes significantly more FLOPs (or vice versa).

To confirm this we conducted a very simple experiment.  We computed
the ordering statistics for 1130 pattern symmetric matrices from the
University of Florida (UF) sparse matrix collection~\cite{DavisHu:2011}
using AMD (2.3.0)~\cite{AmestoyDavisDuff:2004} and METIS
(4.0.3)~\cite{KarypisKumar:1998}.  For 91 of these matrices one
heuristic produced fewer fill elements than the other while performing
worse with respect to the FLOP count at the same time.  For example,
for the matrix ``INPRO/msdoor'' from the UF collection (id 1644), a structural problem,  AMD
produces about 2\% fewer fill elements than METIS while requiring
approximately 22\% more arithmetic operations.

Finally we mention that several approximation algorithms for all three
problems \probname{MinimumFill}, \probname{MinimumFLOPs} and
\probname{Treewidth} exist, e.g.~\cite{AgrawalKleinRavi:1993,
BodlaenderEtal:1995, NatanzonEtal:2000}.

\section{Minimizing FLOPs is NP hard}
\label{sec:flops_np_hard}

We now show that minimizing the FLOP count in sparse Cholesky
factorization is indeed an NP hard problem.  To do so, we reduce the
\probname{MaxCut} problem to a certain class of quadratic arrangement
problems in section \ref{sec:OQA}.  In section \ref{sec:QCC} we reduce
such a quadratic arrangement problem to the minimum FLOPs problem via
a quadratic variation of the bipartite chain graph completion problem.

\subsection{Quadratic vertex arrangement problems}
\label{sec:OQA}

In the optimal linear arrangement
problem, we are given a graph $G =
(V,E)$ and are asked to arrange the vertices of $G$ at positive
integer positions on the real line such that the sum of the implied
edge lengths is minimum:
 
\begin{center}
\framebox{
\begin{minipage}{.85\textwidth}
\probname{OptimalLinearArrangement (OLA)}\\
Instance: Graph $G=(V, E)$ on $n$ vertices, $k\in \N$\\
Question: Is there a bijection $\alpha : V \rightarrow \nset{n}$ s.t.
$\sum_{(u,v) \in E} |\alpha(u) - \alpha(v)| \le k$?
\end{minipage}
}
\end{center}

OLA is NP hard~\cite[GT42]{GareyJohnson:1979}.  It is also known as \probname{MinimumOneSum} (M1S) and minimizes the
1-norm of a vector of distances implied by the linear arrangement of
the vertices of the graph. Other norms have been considered;  for the
2-norm (\probname{MinimumTwoSum}, M2S) and the infinity norm
(\probname{Bandwidth}) the corresponding arrangement problems are
known to be NP hard~\cite{Papadimitriou:1976, GeorgePothen:1997}.  In
contrast to these arrangement problems, the class of arrangement
problems we discuss here cannot be expressed in terms of a $p$-norm of
the distance vector.

\begin{figure}
    \begin{center}
        \includegraphics[width=.7\textwidth]{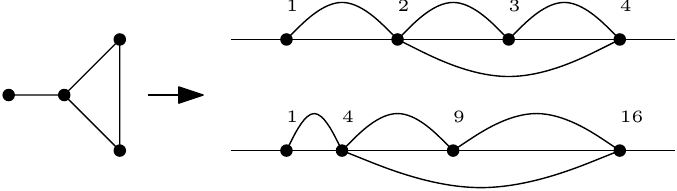}
\caption{Linear and quadratic arrangement of a graph.  The quadratic
cost function here is $f(x) = x^2$, so this is an instance of
OQA(0).  The cost of the linear arrangement is 5, while the quadratic
cost is 27.\label{fig:la_and_qa}}
    \end{center}
\end{figure}

Instead of laying out the vertices of $G$ at equally spaced positions,
we consider certain quadratically spaced positions (see Fig.
\ref{fig:la_and_qa}).  We call this the
\probname{OptimalQuadraticArrangement}($c$) (OQA($c$)) problem.  Let
\begin{equation}
\label{eqn:c_polynomial}
c = c_2 X^2 + c_1 X + c_0,\quad c_0,c_1,c_2 \in \N
\end{equation}
be a polynomial of degree at most 2 with non-negative integer
coefficients.  We regard $c$ as a parameter for the function
\begin{equation*}
    f : \nset{n} \rightarrow \Z_+,\; x \mapsto x^2 + c(n)x.
\end{equation*}
Then the positions on the real line at which we place the vertices of
$G$ are given by $f(\nset{n})$.  Notice that $f$ is a bijection.  Allowing
for a minor abuse of notation we will sometimes write $c$ instead of
$c(n)$ when it can be seen from the context whether the integer $c(n)$
or the polynomial $c$ is referred to.  Formally we define the
following class of decision problems, parametrized by the polynomial
$c$ as follows:
\begin{center}
\framebox{
\begin{minipage}{.8\textwidth}
\probname{OptimalQuadraticArrangement}(c) (OQA($c$))\\ Instance: Graph
$G=(V, E)$ on $n$ vertices, $k\in \N$\\ Question: Is there a bijection
$\alpha : V \rightarrow \nset{n}$ such that $\sum_{(u,v) \in E}
|f(\alpha(u)) - f(\alpha(v))| = |\alpha(u)^2 - \alpha(v)^2 + c
(\alpha(u) - \alpha(v))| \le k$?
\end{minipage}
}
\end{center}
For example, when $c$ is the zero polynomial, this includes the
problem where the vertex positions are laid out according to the
mapping $x \mapsto x^2$.  In section \ref{sec:oqa(c)_is_np_hard} we
will prove that OQA($c$) is NP hard for \emph{every} choice  of the
polynomial $c$ in \eqref{eqn:c_polynomial}.

\subsubsection{Basic properties of the OQA problem}

We will now discuss a few properties of the OQA problem and introduce
some useful notation for later use.  Given a graph $G = (V,E)$ on $n$
vertices, a bijection $\alpha : V \rightarrow N \subset \Z_+$ and the
quadratic function $f(x) = x^2 + c(n)x$, we denote the quadratic cost
of such an arrangement by
\begin{equation*}
    \quadcost(\alpha) \defby \sum_{(u,v) \in E} |f(\alpha(u)) -
        f(\alpha(v))|,
\end{equation*}
and the corresponding linear cost for the arrangement by
\begin{equation*}
    \lincost(\alpha) \defby \sum_{(u,v) \in E} |\alpha(u) - \alpha(v)|.
\end{equation*}
For an edge $e = (u,v) \in E$, we sometimes write its implied quadratic
cost under the ordering $\alpha$ as
\begin{equation*}
    \phi_\alpha(e) \defby |f(\alpha(u)) - f(\alpha(v))|,
\end{equation*}
where we may drop the index $\alpha$ if the ordering is implied by
the context.

\begin{definition}
For a given ordering $\alpha : V \rightarrow
\nset{n}$ and a non-negative integer $r$, we denote by $\alpha+r$ the
following translated ordering:
\begin{align*}
    \alpha+r : V & \rightarrow \{1+r, \dotsc, n+r\}\\
    v & \mapsto \alpha(v) + r
\end{align*}
\end{definition}

Translated orderings are actually not consistent with the definitions
of the arrangements problems (there we required $\alpha$ to map onto $\nset{n}$).
They are compatible with the definitions of $\quadcost(\cdot)$ and
$\lincost(\cdot)$, however.

The linear arrangement cost is translation invariant, since
\begin{equation*}
    \lincost(\alpha+r) = \sum_{(u,v) \in E} |(\alpha(u) + r) - (\alpha(v) + r)|
        = \sum_{(u,v) \in E} |\alpha(u) - \alpha(v)|
        = \lincost(\alpha),
\end{equation*}
but the quadratic arrangement costs of the two orderings are
different;  a translation results in a linear change of the
arrangement cost:
\begin{lemma}[translation lemma] \label{lemma:translation}
For an ordering $\alpha : V \rightarrow \nset{n}$ and a displacement $r \in
\N$ we have
\begin{equation*}
    \quadcost(\alpha+r) = \quadcost(\alpha) + 2r \lincost(\alpha).
\end{equation*}
\end{lemma}
\begin{proof}
We can assume that in the given ordering $\alpha$, we have
for an edge $(u,v) \in E$ that $\alpha(u) < \alpha(v)$ (otherwise call
this undirected edge $(v,u)$ instead).
\begin{equation*}
\begin{split}
    \quadcost(\alpha+r)
    & = \sum_{(u,v) \in E} \left( ((\alpha(v) + r)^2 + c (\alpha(v) + r) )
        - ( (\alpha(u) + r)^2 + c (\alpha(u) + r) ) \right)\\
    & = \sum_{(u,v) \in E} \left( \alpha(v)^2 + 2\alpha(v) r
        - \alpha(u)^2 - 2\alpha(u) r
        + c \alpha(v) - c \alpha(u) \right)\\
    & = \quadcost(\alpha) + 2r\lincost(\alpha).
\end{split}
\end{equation*}
\end{proof}

Denote by $K_s$ the complete graph on $s$ vertices.  Both the
quadratic and linear costs for arranging $K_s$ are independent of the
chosen bijection $\alpha$.  Elementary counting immediately gives that
the linear arrangement cost of $K_s$ is $\frac{1}{6} s (s^2-1)$.  The
quadratic cost is given by the following lemma, whose proof is a
straightforward computation\opt{arxiv}{ (see
Appendix~\ref{app:proof_quadcost_clique_trans})}.
\begin{lemma}
    \label{lemma:quadcost_clique_trans}
    Let $\alpha: V(K_s) \rightarrow \nset{s}$ be an arrangement of $K_s$
    and $r \in \N$, then
    \begin{equation*}
        \quadcost(\alpha+r) = \frac{1}{6} s (s^2 - 1) (2r + c + s + 1).
    \end{equation*}
\end{lemma}

It is easy to see that the OQA problem is different from the OLA
problem in the same sense as \probname{MinimumFill} and
\probname{MinimumFLOPs} are different\opt{arxiv}{ (see Appendix
\ref{app:oqa_vs_ola})}.

\subsubsection{OQA($c$) is NP hard}
\label{sec:oqa(c)_is_np_hard}

We will now show that OQA($c$) is an NP hard problem for every choice of
the polynomial $c$ in \eqref{eqn:c_polynomial}.  Our strategy to
reduce from \probname{MaxCut} follows along the lines of the reduction
from \probname{MaxCut} to OLA in \cite[chap. 8]{Even:1979}, but the
details are very much different.

The reduction will reduce \probname{MaxCut} to the \emph{maximization}
version of OQA.  Thus we show first that maximization and minimization
of the quadratic arrangement are equivalent (in the complexity
sense).
\begin{proposition}
    \probname{MaxOQA(c)} and \probname{MinOQA(c)} are equivalent.
\end{proposition}
\begin{proof}
Let $(G=(V,E); k),\; \card{V} = n,$ be an instance of
\probname{MaxOQA(c)} and define $(\bar{G}; k' \defby \frac{1}{6} n
(n^2 - 1) (c + n + 1) - k)$ to be an instance of \probname{MinOQA(c)}
($\overline{G}$ is the complement of $G$).  Denote by $\overline{E}$
the set of edges of $\overline{G}$, then by Lemma
\ref{lemma:quadcost_clique_trans} (with $r=0$) we know that for any ordering $\alpha : V
\rightarrow \nset{n}$ we have
\begin{equation*}
    \sum_{e \in E} \phi(e) + \sum_{e \in \overline{E}} \phi(e)
        = \frac{1}{6} n (n^2 - 1) (c + n + 1) = k + k',
\end{equation*}
so
\begin{equation*}
    \sum_{e \in E}    \phi(e) \ge k
    \Leftrightarrow
    \sum_{e \in \overline{E}} \phi(e) \le k',
\end{equation*}
which completes the proof.
\end{proof}

From now on, we only consider the maximization version of OQA($c$).

If $G = (V, E)$ is a graph and $X \subseteq V$, we denote by
$\ecut{X}$ the edge cut
\begin{equation*}
    \{(u,v) \in E \setcond u \in X \wedge v \in V\setminus X\}.
\end{equation*}
Sometimes we simply write $\overline{X}$ for $V \setminus X$.  Deciding
whether $G$ admits a cut of size $k \in \Z_+$ or greater, the
\probname{MaxCut} problem, is a fundamental NP complete problem.

We introduce the following notation that we will use in the next two
lemmas and the theorem that follows.  Let $\alpha : V \rightarrow \nset{n}$
be an arrangement for $G= (V,E)$.  For $1 \le j \le n$, we define the
set
\begin{equation*}
    X_j = \{v \in V \setcond \alpha(v) \le j \}.
\end{equation*}
The sets $X_j$ naturally induce cuts $\ecut{X_j}$.

In the reduction from \probname{MaxCut} we will need to rearrange
isolated vertices in a given ordering.  The following two lemmas give
sufficient conditions for performing these rearrangements without
decreasing the arrangement costs.

\begin{figure}
    \begin{center}
        \includegraphics[width=.5\textwidth]{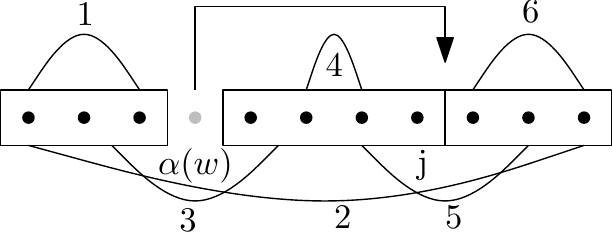}
        \caption{Illustration for the proof of Lemma \ref{lemma:move_iso_right}
        \label{fig:move_iso_right}}
    \end{center}
\end{figure}

\begin{lemma} \label{lemma:move_iso_right}
Let $1 \le  j < n$ and let $w \in V$ be an isolated vertex such that
$\alpha(w) < j$ and $\card{\ecut{X_k}} \le \card{\ecut{X_j}}$ for all
$\alpha(w) \le k \le j$.  Then for the ordering $\alpha' : V
\rightarrow \nset{n}$ defined by
\begin{equation*}
    \alpha'(v) =
    \begin{cases}
        \alpha(v) & \text{if $\alpha(v) < \alpha(w)$ or $j < \alpha(v)$}\\
        j & \text{if $v=w$}\\
        \alpha(v) - 1 & \text{if $\alpha(w) < \alpha(v) \le j$}
    \end{cases}
\end{equation*}
we have $\quadcost(\alpha') \ge \quadcost(\alpha)$.
\end{lemma}
\begin{proof}
For an edge $e = (u,v) \in E$ we may assume that $ \alpha(u)
< \alpha(v)$. We denote the contribution of an edge $e$ to the change
of cost by $\Delta(e) \defby \phi_{\alpha'}(e) - \phi_\alpha(e)$.
Based on the positions of $u$ and $v$ in $\alpha$ relative to
$\alpha(w)$ and $j$, we now calculate $\Delta(e)$; there are six
cases to be considered (see Fig. \ref{fig:move_iso_right}).
\begin{align*}
    \alpha(u) < \alpha(v) < \alpha(w)
        & \Rightarrow \Delta(e) = 0\\
    \alpha(u) < \alpha(w) \wedge j < \alpha(v)
        & \Rightarrow \Delta(e) = 0\\
    \alpha(u) < \alpha(w) < \alpha(v) \le j
        & \Rightarrow \Delta(e) = -(2 \alpha(v) + c - 1)\\
    \alpha(w) < \alpha(u) < \alpha(v) \le j
        & \Rightarrow \Delta(e) = +(2\alpha(u) + c - 1) - (2\alpha(v) + c - 1)\\
    \alpha(w) < \alpha(u) \le j < \alpha(v)
        & \Rightarrow \Delta(e) = +(2\alpha(u) + c - 1)\\
    j < \alpha(u) < \alpha(v)
        & \Rightarrow \Delta(e) = 0
\end{align*}

We now quantify the global change of cost. For accounting purpose, it
is useful to associate a change of cost $\pm (2 \alpha(x) + c - 1)$
with the vertex $x$ (all cost changes are of that form).  Notice that
only vertices $x \in V$ with $\alpha(w) < \alpha(x) \le j$ can have
associated a change of cost with them.  Moving from the $j$-th
position in the arrangement to the left back to position
$\alpha^{-1}(w) + 1$, we pick up a positive change at vertex $x$ if
and only if $\card{\ecut{X_{\alpha(x)}}} >
\card{\ecut{X_{\alpha(x)-1}}}$ and a negative change if and only if
$\card{\ecut{X_{\alpha(x)}}}  < \card{\ecut{X_{\alpha(x)-1}}}$.  If
the size of the cut does not change at $x$, neither does the cost
change (changes may cancel at that vertex though).

Since none of the cuts on the left of $j$ exceeds the size of the cut
$\ecut{X_j}$ and the absolute value of each change is strictly
decreasing as we move to the left, the sum of accumulated changes
stays non-negative throughout until we reach position $\alpha(w) + 1$.
But by reaching that position we have accounted for all changes due to
the reordering, so we have $\quadcost(\alpha') \ge \quadcost(\alpha)$.
\end{proof}

Lemma \ref{lemma:move_iso_right} describes circumstances that allow
moving a single isolated vertex from the left into a locally largest
cut without decreasing the arrangement costs.  Unfortunately, moving
isolated vertices from the right of that cut is not as easy.  In fact
the cost can decrease if we move such a single isolated vertex in a
position where it intersperses the cut\opt{arxiv}{ (see Appendix
\ref{app:ex_move_left})}.  But there are conditions under which we
can move a \emph{block} of isolated vertices from the right as the
following lemma shows.

\begin{figure}
    \begin{center}
        \includegraphics[width=.6\textwidth]{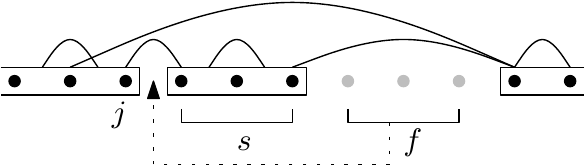}
        \caption{Illustration for the proof of Lemma
        \ref{lemma:move_block_left}.  Edges symbolize the edge classes~$E_i$.
	\label{fig:move_block_left}}
    \end{center}
\end{figure}

\begin{lemma}\label{lemma:move_block_left}
Let $j,s,f \in \N$ be such that $1 \le j < j+s < j+s+f \le n$,
$\card{\ecut{X_j}} > \card{\ecut{X_{j+k}}}$ for $1\le k \le s+f$ and
$\{\alpha^{-1}(j+s+1), \dotsc, \alpha^{-1}(j+s+f)\} \subset V$ are
isolated vertices.  Define the ordering $\alpha'$ by
\begin{equation*}
    \alpha'(v) =
    \begin{cases}
        \alpha(v) & \text{if $\alpha(v) \le j$ or $\alpha(v) > j+s+f$}\\
        \alpha(v) - s &  \text{if $j+s < \alpha(v) \le j+s+f$}\\
        \alpha(v) + f &  \text{if $j < \alpha(v) \le j+s$}.
    \end{cases}
\end{equation*}
If $j+1 + \frac{c+f}{2} \ge \card{\ecut{X_{j+s+f}}}(s-1)$ then we have
$\quadcost(\alpha') \ge \quadcost(\alpha)$.
\end{lemma}
\begin{proof}
As in Lemma \ref{lemma:move_iso_right} we denote the change of cost
when passing from $\alpha$ to $\alpha'$ for an edge $e = (u,v) \in E$
by $\Delta(e)$ and we assume that $\alpha(u) < \alpha(v)$. Based on
the positions of the end points, the edges can be divided into
six disjoint sets (see Fig.~\ref{fig:move_block_left}):
\begin{align*}
    E_1 & \defby \{ (u,v) \in E \setcond \alpha(u) < \alpha(v) \le j\}\\
    E_2 & \defby \{ (u,v) \in E \setcond
        \alpha(u) \le j \wedge j+s+f < \alpha(v) \}\\
    E_3 & \defby \{ (u,v) \in E \setcond
        \alpha(u) \le j < \alpha(v) \le j+s \}\\
    E_4 & \defby \{ (u,v) \in E \setcond
        j < \alpha(u) < \alpha(v) \le j+s \}\\
    E_5 & \defby \{ (u,v) \in E \setcond
        j < \alpha(u) \le j+s < j+s+f < \alpha(v) \}\\
    E_6 & \defby \{ (u,v) \in E \setcond
        j+s+f < \alpha(u) < \alpha(v) \}
\end{align*}
From the definition of $\alpha'$, we see that $\Delta(e) = 0$, for $e
\in E_1 \cup E_2 \cup E_6$.  For the other three cases a short
calculation shows that
\begin{align*}
    e \in E_3 \Rightarrow \Delta(e) & = f(2 \alpha(v) + c + f),\\
    e \in E_4 \Rightarrow \Delta(e) & = 2 f (\alpha(v) -
    \alpha(u))\quad\text{and}\\
    e \in E_5 \Rightarrow \Delta(e) & = - f(2 \alpha(u) + c + f).
\end{align*}

We now derive a lower bound for the cost difference of $\alpha'$ and
$\alpha$.  We will use that
\begin{equation*}
    \card{E_3} - \card{E_5} = \card{E_3} +
        \card{E_2} - (\card{E_5}+ \card{E_2}) = \card{\ecut{X_j}} -
        \card{\ecut{X_{j+s+f}}} \ge 1,
\end{equation*}
as well as $\card{E_5} \le \card{\ecut{X_{j+f+s}}}$.
We immediately drop the non-negative contribution from
edges in $E_4$ and calculate
\begin{equation*}
\begin{split}
\quadcost(\alpha') - \quadcost(\alpha) & \ge
    \sum_{(u,v) \in E_3} f (2 \alpha(v) + c + f)
    - \sum_{(u,v) \in E_5} f (2 \alpha(u) + c + f)\\
& \ge \card{E_3} f (2(j+1) + c + f) - \card{E_5} f  (2 (j+s) + c + f^2)\\
& = (\card{E_3} - \card{E_5}) f (2 (j+1) + c + f) - \card{E_5} 2 f (s-1)\\
& \ge f(2(j+1) + c + f) - \card{\ecut{X_{j+s+f}}} 2 f (s-1)\\
& = f ( 2 (j + 1) + c + f - \card{\ecut{X_{j+s+f}}} 2 (s-1)).
\end{split}
\end{equation*}
By assumption we have $j+1 + \frac{c+f}{2} \ge
\card{\ecut{X_{j+s+f}}}(s-1)$, so the difference $\quadcost(\alpha') -
\quadcost(\alpha)$ is non-negative.
\end{proof}

\begin{theorem}
Let $c = c_2 X^2 + c_1 X + c_0$ be a polynomial of degree at most two
with non-negative integer coefficients. Then
$\probname{MaxCut} \propto \probname{OQA}(c)$.
\end{theorem}
\begin{proof}
Let $(G'=(V',E'); k')$ be an instance of \probname{MaxCut}.  We define
an instance $(G=(V,E); k)$ for \probname{OQA} by adding $n^5$
isolated vertices to $G'$:  Let $W$ be set of size $n^5$, then we set
\begin{equation*}
V = V' \cup W, \quad E = E'\quad \text{and}\quad k = n^{10} k'.
\end{equation*}

Assume that $G'$ admits a cut $\ecut{X'}$ of size at least $k'$.  We
define an ordering $\alpha : V \rightarrow \nset{n + n^5}$ for $G$ by
\begin{equation*}
\begin{split}
\alpha(X') & = \{1,\dotsc, \card{X'}\}\\
\alpha(W) & = \{\card{X'} + 1, \dotsc, \card{X'} + n^5\}\\
\alpha(V' \setminus X') & = \{\card{X'} + n^5 + 1, \dotsc, n^5 + n\},
\end{split}
\end{equation*}
where the ordering within the sets $X', W$ and $V\setminus X'$ is
arbitrary.  We now derive a lower bound for $\quadcost(\alpha)$:
Every edge $e \in \ecut{X'}$ induces a cost of at least
\begin{equation*}
\begin{split}
\phi(e) & \ge (n^5 + 2)^2 + c(n^5 + 2) - 1^2 - c\cdot1\\
& = n^{10} + (4+c)n^5 + c + 3,
\end{split}
\end{equation*}
so
\begin{equation*}
\begin{split}
\quadcost(\alpha) & = \sum_{e \in E} \phi(e)
        \ge \sum_{e \in \ecut{X'}} \phi(e)
        \ge  (n^{10} + (4+c) n^5 + c + 3) \card{\ecut{X'}}\\
    & \ge n^{10} k' = k.
\end{split}
\end{equation*}

For the reverse direction assume that we are given an ordering $\alpha
: V \rightarrow \nset{n + n^5}$ such that $\quadcost(\alpha) \ge k$.  In
order to show that $G'$ has a cut of size at least $k'$, we will first
rearrange $\alpha$, without decreasing the ordering cost, so that the
vertices in $W$ are ordered consecutively.  This reordering process
has two stages: First, using Lemma \ref{lemma:move_iso_right}, we will
move isolated vertices to the right so that they intersperse with
locally largest cuts.  This yields a block structure of isolated
vertices of $W$ to which we will then apply Lemma
\ref{lemma:move_block_left} in a second step.

For the first stage, let $b_1$ be the largest index of a maximum cut
among the cuts $\ecut{X_i}$, that is,
\begin{equation*}
    b_1 = \max \{\argmax_{1\le i\le n^5+n} \card{\ecut{X_i}}\}.
\end{equation*}
Among the $b_1$ vertices in $X_{b_1}$ denote by $n_1$ the number of
vertices from $V'$ and by $f_1$ the number of vertices from W, so
$n_1 = b_1 + f_1$.  By Lemma \ref{lemma:move_iso_right}, we can
rearrange $\alpha$ so that $\alpha^{-1}(\{1,\dotsc,n_1\}) \subseteq
V\setminus W$ and $\alpha^{-1}(\{n_1 + 1, \dotsc, n_1 + f_1\})
\subseteq W$ without decreasing the cost.

Iterating this procedure on the vertices ordered after $b_1$, we
obtain an ordering in which the vertices appear partitioned in $h$
parts, where in each part the vertices of $V'$ and $W$ are ordered
consecutively (see Fig.  \ref{fig:block_structure}).  More formally,
the ordering has the following properties:
\begin{gather*}
    0 \bydef b_0 < b_1 < b_2 < \dotsb < b_h = n + n^5,\\
    b_k = \max \{\argmax_{b_{k-1} < i\le n^5+n}
\card{\ecut{X_i}}\},\quad 1\le k \le h,\\
    \card{\ecut{X_{b_1}}} > \card{\ecut{X_{b_2}}} > \dotsb
        > \card{\ecut{X_{b_h}}} = 0,\\
    n_k + f_k = b_k - b_{k-1}, \quad 1 \le k \le h,\\
    \sum n_k = n, \quad \sum f_k = n^5,\\
    \alpha^{-1}(\{b_{k-1} + 1, \dotsc, b_{k-1} + n_k \})
        \subseteq V\setminus W,\quad 1 \le k \le h,\\
    \alpha^{-1}(\{b_{k-1} + n_k + 1, \dotsc, b_{k-1} + n_k + f_k \})
        \subseteq W,\quad 1 \le k \le h
\end{gather*}

\begin{figure}
    \begin{center}
        \includegraphics[width=.6\textwidth]{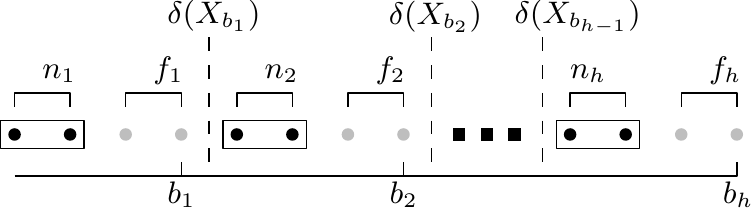}
        \caption{Illustration for the block structure arising from
        moving isolated vertices closest to their rightmost largest cut.
        \label{fig:block_structure}}
    \end{center}
\end{figure}

Note that some of the $f_k$ may be zero but all $n_k > 0$.  Since
$\card{\ecut{X_{b_{1}}}}$ is trivially bounded by the linear cutwidth
of the complete graph on $n$ vertices and the size of the cuts
$\ecut{X_k}$ is strictly decreasing, we obtain $h \le \frac{n^2}{4}$.

Now begins the second stage of the rearrangement.  From the given
block structure, we will perform a series of rearrangements using
Lemma \ref{lemma:move_block_left} until eventually all vertices from
$W$ intersperse between the sets $X_{b_1}$ and $\overline{X_{b_1}}$.
Each of the reordering operations will maintain the block structure as
a whole, but the individual values of the $f_k$ will change.  In order
to simplify notation, we will not explicitly distinguish between
different orderings $\alpha$ and values $f_k$'s at the different
stages during the process.

Let $\nu \in \argmax_{1 \le k \le h} f_k$; since $\sum f_k = n^5$ and $h
\le \frac{n^2}{4}$, we have $f_\nu \ge 4n^3$.  Define $j \defby b_{\nu
- 1}, s \defby n_\nu, f \defby f_\nu$.  By construction we have that
$\card{\ecut{X_j}} > \card{\ecut{X_{j+k}}}$ for $1\le k \le s + f$ and
\begin{equation*}
    j + 1 + \frac{c+f}{2} \ge \frac{f}{2} \ge 2n^3 \ge \frac{n^2}{4} (n-1) \ge
        \card{\ecut{X_{j+s+f}}}(s-1).
\end{equation*}
So the assumptions of Lemma \ref{lemma:move_block_left} are met and in
the rearranged ordering we now have $f_{\nu - 1} \ge 4n^3$ and $f_\nu
= 0$.  By induction we obtain an ordering in which the block structure
satisfies $f_1 \ge 4n^3$ and $f_2 = \dotsb = f_\nu = 0$.

Next set $j\defby b_1 \ge 4n^3$, $s \defby \sum_{k=2}^\nu (n_k + f_k)
+ n_{\nu + 1} = \sum_{k=2}^{\nu + 1} n_k  \le n$ and $f \defby f_{\nu
+ 1}$.  By construction we have that $\card{\ecut{X_j}} >
\card{\ecut{X_{j+k}}}$ for $1\le k \le s + f$ and
\begin{equation*}
    j + 1 + \frac{c+f}{2} \ge j \ge 4n^3 \ge \frac{n^2}{4} (n-1)
        \ge \card{\ecut{X_{j+s+f}}}(s-1).
\end{equation*}
This permits us to apply Lemma \ref{lemma:move_block_left} and in the
rearranged ordering we now have $f_{\nu + 1} = 0$ while $f_1 \ge 4n^3$
is maintained.  By induction we arrive at an ordering where $f_2 =
\dotsb = f_h = 0$, which implies $f_1 = n^5$.  Denote this final
ordering by $\alpha'$.  Since none of the reordering operations has
ever decreased the total arrangement cost, we have $\quadcost(\alpha')
\ge \quadcost(\alpha) \ge k$, where $\alpha$ is the very original
ordering that we started with.

Next we derive an upper bound for $\quadcost(\alpha')$.  We
classify the edges of $G$ in three different categories and bound the
contribution from each of these sources.
\begin{asparaenum}

\item If $(u,v) \in X_1 \times X_1$, then the total cost of these edges
is strictly bounded by the arrangement cost of a clique of size $n$
being ordered at positions $1,\dotsc, n$.  By Lemma
\ref{lemma:quadcost_clique_trans} (with $r=0$), this cost is $\frac{1}{6} n(n^2 - 1)(c+n+1)$.

\item If $e = (u,v) \in X_1 \times \overline{X_1}$, then the cost
implied by $e$ is at most $(n^5 + n)^2 + c(n^5 + n)- 1^2 - c$.

\item If $(u,v) \in \overline{X_1} \times \overline{X_1}$, then the
total cost of these edges is strictly bounded by the arrangement cost
of a clique of size $n$ being ordered at positions $n^5 + 1,\dotsc,
n^5 + n$.  By Lemma \ref{lemma:quadcost_clique_trans}, this cost is
$\frac{1}{6} n (n^2 - 1)(2n^5 + c + n+1)$.
\end{asparaenum}
In total we obtain
\begin{align*}
    n^{10} k' & = k \le \quadcost(\alpha) \le \quadcost(\alpha')
        \displaybreak[1]\\
    & \le \card{\ecut{X_1}} ((n^5 + n)^2 + c(n^5 + n)- 1 - c)
        + \frac{n}{6}(n^2 -1) (c + n+1)\\
        & \quad + \frac{n}{6} (n^2 - 1) (2n^5 + c + n + 1)
        \displaybreak[1]\\
    & \le \card{\ecut{X_1}} (n^{10} + 2n^6 + cn^5 + n^2 + cn) 
        + \frac{1}{3} n(n^2) (n^5 + c + n + 1)
        \displaybreak[1]\\
    & \Rightarrow k' \le \card{\ecut{X_1}} +
        \underbrace{\card{\ecut{X_1}}
        \frac{2n^6 + cn^5 + n^2 + cn}{n^{10}}
            + \frac{n^8 + cn^3 + n^4 + n^3}{3n^{10}}}_{\bydef r(n)}.
\end{align*}
Since $\card{\ecut{X_1}} \le \frac{n^2}{4}$, we have
\begin{equation*}
r(n) \le \frac{1}{2n^2} + \frac{c}{4n^3} + \frac{1}{4n^6} + \frac{c}{4n^7}
    + \frac{1}{3n^2} + \frac{c}{3n^7} + \frac{1}{3n^6} + \frac{1}{3n^7}.
\end{equation*}
Because $c$ is a polynomial of degree at most two, there exists an
integer $n_c \in \N$ such that
\begin{equation*}
r(n) < 1,\quad \text{for all}\quad n \ge n_c.
\end{equation*}
Together with the integrality of $\card{\ecut{X_1}}$ and $k'$, it
follows that $\card{\ecut{X_1}} \ge k'$.
\end{proof}

\subsection{Reduction from OQA to the minimum FLOPs problem}
\label{sec:QCC}

In this section we reduce OQA(c) to the minimum FLOPs problem for a
certain polynomial $c$. Our strategy follows the pattern that
Yannakakis used for the reduction of OLA to minimum
fill~\cite{Yannakakis:1981}, but again the details are much different.
In particular we employ a quadratic variation of the bipartite chain
graph completion problem, which we discuss in section
\ref{sec:reduction_from_qcc}.  In section \ref{sec:reduction_from_oqa}
we give a reduction from OQA(c) to this quadratic chain completion
problem.

\subsubsection{Reduction from bipartite quadratic chain completion}
\label{sec:reduction_from_qcc}

Let $G = (P, Q, E)$ be a bipartite graph on $p+q$ vertices, $p
\defby \card{P}, q \defby \card{Q}$.  Recall that for a vertex $v \in
P$ we denote its neighbourhood in $G$ by $\nhd{v}$.  $G$ is a
\emph{bipartite chain graph} if there exists a bijection $\alpha : P
\rightarrow \nset{p}$ such that
\begin{equation}
    \nhd{\alpha^{-1}(i)} \supseteq \nhd{\alpha^{-1}(i+1)},\;
    1 \le i \le p-1.
\end{equation}
Note that $G$ admits such a chain ordering for $P$ if and only if $G$
admits a chain ordering for $Q$, so the definition does not depend on
a particular partition of $G$.  For a bipartite graph, the property of
being a chain graph is hereditary and the minimal obstruction set is
$\{2K_2\}$ \cite[Lemma 1]{Yannakakis:1981}.

Yannakakis considers the problem of completing a given bipartite graph
into a bipartite chain graph.  We formulate the corresponding decision
problem in terms of vertex degrees:

\begin{center}
\framebox{
\begin{minipage}{.8\textwidth}
\probname{BipartiteChainCompletion (BCC)}\\
Instance: Bipartite graph $G=(P, Q, E), k\in \N$\\
Question: Is there a set of edges $F \subseteq P \times Q$ such that
$G^+ = (P, Q, E \cup F)$ is a chain graph and 
$\sum_{v \in P} \vdeg[G^+]{v} \le k$ ?
\end{minipage}
}
\end{center}

Note that our metric of measuring the cost of the chain completion is
equivalent to minimizing $\card{F}$ in the formulation above, because
\begin{equation*}
\sum_{v \in P} \vdeg[G^+]{v} = \card{E} + \card{F}.
\end{equation*}

Our quadratic variation of the bipartite chain completion problem has
a cost function which is a quadratic function of the vertex degrees in
the augmented graph.
\begin{center}
\framebox{
\begin{minipage}{.8\textwidth}
\probname{QuadraticChainCompletion (QCC)}\\ Instance: Bipartite graph
$G=(P, Q, E)$ on $p+q$ vertices ($p=\card{P}$, $q=\card{Q}$) where the partition $P$ is
designated, $k\in \N$\\ Question: Is there a set of edges $F
\subseteq P \times Q$ such that $G^+ = (P, Q, E \cup F)$ is a
chain graph with
\begin{equation*}
    \qcc(F) \defby
    \sum_{v \in P} \vdeg[G^+]{v}^2 + 2(p+1) \sum_{v\in P}
\vdeg[G^+]{v} \le k?
\end{equation*}
\end{minipage}
}
\end{center}

Unlike for BCC, it is not clear whether the minima of our quadratic
variation depend on the particular vertex partition chosen, which is
why the information which partition to consider is part of the input.
Of course, the particular cost value (defined by $\qcc$) of a
bipartite chain graph embedding depends on the partition (for
example, consider the simple path on three vertices).

The reduction from BCC to \probname{MinimumFill} in
\cite{Yannakakis:1981} involves a construction that relates certain
triangulations to chain embeddings, which we adapt to our needs by
augmenting it with an additional vertex set $U$:

\begin{definition}\label{def:c_of_g}
Let $G = (P, Q, E)$ a bipartite graph on $p+q$ vertices and $U
= \{u_v \setcond v \in P\}$ a set of $p$ vertices. We define the graph
$C = C(G) = (V', E')$ by
\begin{align*}
    V' & = P \cup Q \cup U\\
    E' & = E \cup (P \times P) \cup ((Q \cup U) \times (Q \cup U)) \cup
        \{(v, u_v) \setcond v \in P\}.
\end{align*}
Further, for a given bijection $\alpha : P \rightarrow \nset{p}$, we define
the set
\begin{equation*}
    G(\alpha) =
        \{(\alpha^{-1}(i), u_{\alpha^{-1}(j)}) \setcond 1 \le i < j \le p\}
        \subset P \times U.
\end{equation*}
\end{definition}

\begin{figure}[t]
    \begin{center}
        \subfigure[$G=(P,Q,E)$]{%
            \includegraphics{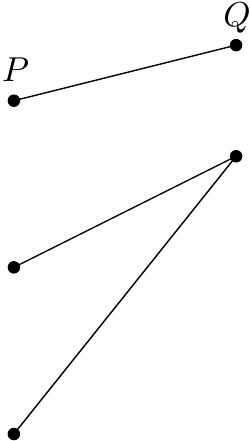}%
            \label{fig:qcc_to_minflops-a}%
        }%
        \hfill
        \subfigure[$C=C(G)$]{%
            \includegraphics{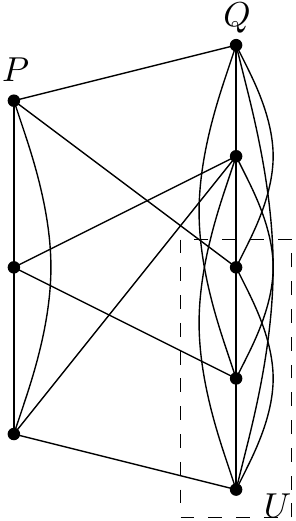}%
            \label{fig:qcc_to_minflops-b}%
        }%
        \hfill
        \subfigure[$C^+$, a triangulation of $C$]{%
            \includegraphics{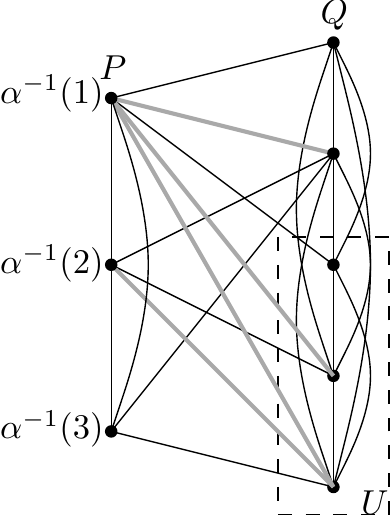}%
            \label{fig:qcc_to_minflops-c}%
        }%
    \end{center}
    \caption{Illustration for the reduction from QCC to minimum FLOPs.
    Figs.~\ref{fig:qcc_to_minflops-a} and \ref{fig:qcc_to_minflops-b} show
    the construction of $C(G)$. Notice that the two topmost vertices
    of $P$ and the vertices of $Q$ induce a $2K_2$ in $G$ and a
    chordless cycle in $C$.  Fig.~\ref{fig:qcc_to_minflops-c} shows a
    triangulation of $C$; the fill edges are shown in gray.  The
    topmost fill edge turns $G$ into a chain graph with $P$-chain
    ordering $\alpha$, the other three fill edges constitute
    $G(\alpha)$. $\alpha^R$ is a prefix of a PEO for $C^+$.
    \label{fig:qcc_to_minflops}}
\end{figure}

Figs.~\ref{fig:qcc_to_minflops-a} and \ref{fig:qcc_to_minflops-b}
give an example for the construction of $C(G)$.  The next lemmas describe how
chain completions of $G$ relate to triangulations of $C(G)$ and
$G(\alpha)$, giving an analogon to~\cite[Lemma 2]{Yannakakis:1981}.
Fig.~\ref{fig:qcc_to_minflops-c} illustrates this relationship.

\begin{definition}
Let $M$ be a set of $m$ elements and $\alpha : M \rightarrow \nset{m}$
a bijection.  Then the reverse bijection $\alpha^R : M \rightarrow
\nset{m}$ is uniquely defined by the property $\alpha^{-R}(i) \defby
(\alpha^{R})^{-1}(i) = \alpha^{-1}(m-i+1)$ for $1\le i \le m$.
\end{definition}

For the following we recall that a \emph{minimal} triangulation for a
graph is an \emph{inclusion minimal} set of edges whose addition
yields a chordal graph.  Analogously we will speak of minimal chain
completions for a given bipartite graph.  There is no loss of
generality if we assume that the decision problems from above are
restricted to minimal completions.  Recall also that a PEO for a graph
$G = (V,E)$ is a bijection $\alpha : V \rightarrow \nset{n}$, $n =
\card{V}$, such that eliminating vertices in the order implied by
$\alpha^{-1}$ does not cause any fill.  By a \emph{prefix} of a PEO
$\alpha$ we mean a restriction $\alpha_{|_W}$ for some $W \subset V$
such that $\alpha^{-1}(k) = \alpha^{-1}_{|_W}(k)$, for $1 \le k \le
\card{W}$.

\begin{lemma} \label{lemma:charac_minimal_triangulation}
Let $G=(P, Q, E)$ be a bipartite graph, $C = C(G) = (V', E') =
(P \cup Q \cup U, E')$ and $F' \subseteq V' \times V'$ a minimal
triangulation of $C$.  Set $F'_U \defby F' \cap (P \times U), F'_Q
\defby F' \cap (P \times Q)$.  Then there exists a bijection $\alpha :
P \rightarrow \nset{p}$ such that
\begin{asparaenum}[i)]
    \item $F'_U = G(\alpha)$\label{enum:fill_u_part},
    \item $(P, Q, E \cup F'_Q)$ is a chain graph and admits
    $\alpha$ as a chain ordering for $P$.\label{enum:fill_q_part}
\end{asparaenum}
\end{lemma}
\begin{proof}
Since $P$ and $Q \cup U$ are already cliques in $C$, we have $F'
\subseteq P \times (Q \cup U)$, so $F' = F_U' \cup F_Q'$ is a
partitioning of $F'$.  Since $F'$ is minimal, there exists a PEO
$\beta$ for $C^+$ such that $C_\beta^+ = C^+$, and because $Q\cup U$
is a clique in $C^+$, we can choose $\beta$ so that it orders $Q\cup
U$ last \cite[Corollary 4]{Rose:1972}, that is,
\begin{equation*}
    \beta^{-1} (\{1,\cdots,p\}) = P, \;
    \beta^{-1} (\{p+1, \cdots, 2p+q\}) = Q \cup U.
\end{equation*}

Denote by $N_j$ the neighborhood of the vertex $\beta^{-1}(j)$ in the
reduced elimination graph at step $j$, and $F'_j$ the set of fill edges
introduced at step $j$ that are incident with $U$. We will show the
following statement by induction (for $1\le j \le p$):  In the $j$-th
elimination step, we have
\begin{align*}
    N_j \cap (P \cup U)
        & = \{\beta^{-1}(i) \setcond j < i \le p\}
        \cup \{u_{\beta^{-1}(i)} \setcond 1 \le i \le j\},\\
    F'_j & = \{ (\beta^{-1}(i), u_{\beta^{-1}(j)}) \setcond j < i \le p\}.
\end{align*}

By inspection of the graph $C$ we find that the statement is true for
$j=1$.  Next assume that the statement is true for all $k$  with $1
\le k < j$.  By the induction assumption, the fill edges incident with
$U$ introduced up
to step $j$ are
\begin{equation}
    \bigcup_{k=1}^{j-1} F'_k
        = \bigcup_{k=1}^{j-1}
        \{ (\beta^{-1}(i), u_{\beta^{-1}(k)}) \setcond k < i \le p \}.
    \label{eqn:fill_up_to_j}
\end{equation}
So at the elimination step $j$, the set of vertices of $U$ that the
vertex $\beta^{-1}(j) \in P$ is adjacent to because of any prior fill
edge is $\{u_{\beta^{-1}(i)} \setcond 1 \le i < j\}$, so we obtain
\begin{equation*}
    N_j \cap (P\cup U) = \{ \beta^{-1}(i) \setcond j < i \le p \}
        \cup \{ u_{\beta^{-1}(i)} \setcond 1 \le i \le j\}.
\end{equation*}
Since the edges \eqref{eqn:fill_up_to_j} are already present at step
$j$, the only edges that need to be added in order to turn this set of
vertices into a clique 
\begin{equation*}
    \{ (\beta^{-1}(i), u_{\beta^{-1}(j)}) \setcond j < i \le p\}
    = F'_j,
\end{equation*}
which completes the proof of the claim.

Let $\alpha \defby (\beta|_P)^R$.  Noting that $F'_p = \emptyset$, it
follows from the claim that
\begin{align*}
    F' \cap (P\times U) & = \bigcup_{j=1}^{p-1} F'_k
    = \bigcup_{j=1}^{p-1}
        \{ (\beta^{-1}(i), u_{\beta^{-1}(j)}) \setcond j < i \le p \}\\
    & = \{ (\alpha^{-1}(i), u_{\alpha^{-1}(j)}) \setcond 1 \le i < j \le p \}
    = G(\alpha).
\end{align*}

Now we have constructed $\alpha$ and shown \eqref{enum:fill_u_part}.
To show \eqref{enum:fill_q_part}, note that $P$ is a clique in
$C(G)$ and $\alpha^R = \beta|_P$ is also a prefix of a PEO for the
induced subgraph $C^+[P\cup Q]$.  So by the construction of
$C$, $\alpha$ is a chain ordering for $P$ in in $(P, Q, E \cup
F_Q')$.
\end{proof}

The previous lemma characterizes minimal triangulations of $C(G)$:
They decompose into a chain completion for $G$ and a set $G(\alpha)$
such that $\alpha$ is a compatible chain ordering.  The next two
lemmas give a reverse direction, so every triangulation of $C(G)$
uniquely defines a chain completion of $G$ and vice versa.

\begin{lemma}[chordal patching lemma, folklore]\label{lemma:patching}
Let $G=(V,E)$ be a graph where the vertices are partitioned in three
disjoint sets $V = A \cup B \cup C$.  Then $G$ is chordal if the
following three conditions are satisfied:
\begin{enumerate}
    \item $G[V\setminus C]$ has two connected components $A, B$,
    \item $G[C]$ is a clique,
    \item $G[A\cup C]$ and $G[B\cup C]$ are chordal.
\end{enumerate}
\end{lemma}
\begin{proof}
Let $Z$ be a simple cycle of length at least $4$ in $G$.  If $Z$ is
entirely contained in $A\cup C$ or $B\cup C$, then $Z$ has a chord.
Otherwise, $Z$ contains vertices both of $A$ and $B$, so $Z$ intersects
$C$ at least at two non-consecutive vertices of $Z$, which gives a
chord in $Z$ since $C$ is a clique.
\end{proof}

\begin{lemma}\label{lemma:chain_completion_to_triangulation}
Let $G = (P, Q, E)$ be a bipartite graph and let $F \subseteq P\times
Q$ such that $G^+ = (P, Q, E \cup F)$ admits $\alpha : P
\rightarrow \nset{p}$ as a chain ordering.  Then $F' = F \cup G(\alpha)$ is
a triangulation for $C = C(G) = (V', E')$ and $\alpha^R$ is a prefix
of a PEO for $C^+ = (V', E' \cup F')$.
\end{lemma}
\begin{proof}
Let $C^+_Q  = C^+[P \cup Q]$ and $C^+_U = C^+[P \cup U]$.  We first
show that $C^+_Q$ and $C^+_U$ are chordal.  A chordless cycle in
$C^+_Q$ implies an induced subgraph in $G^+$ isomorphic to $2K_2$,
which contradicts the assumption that $G^+$ is a bipartite chain
graph.  So $C^+_Q$ is chordal.

From the definition of $G(\alpha)$ it follows that we can use
$\alpha^R$ to carry out $p$ steps of vertex elimination in $C^+_U$
without introducing a fill edge.  But after these $p$ steps only a
clique of size $p$ remains, so $C^+_U$ admits a PEO
which implies that $C^+_U$ is chordal.

Noting that $P$ is a clique in $C^+$, it follows from Lemma
\ref{lemma:patching} that $C^+$ is chordal.  Since $G^+$ is a chain
graph and since $P$ is a clique in $C^+$, no fill edge is introduced
when eliminating along $\alpha^R$.  Consequently, $\alpha^R$ is a
prefix of a PEO for $C^+$.
\end{proof}

The set $G(\alpha)$ in any triangulation $C^+$ of $C(G)$ simplifies
the FLOP counting in the reduction from
\probname{QuadraticChainCompletion} as we will see now.  

\begin{theorem}
\probname{QuadraticChainCompletion} $\propto$ \probname{MinimumFLOPs}.
\end{theorem}
\begin{proof}
As before, we continue to use the notation from
Definition~\ref{def:c_of_g}.  By Lemmas
\ref{lemma:charac_minimal_triangulation} and
\ref{lemma:chain_completion_to_triangulation} every chain completion
$F$ of $G$ gives a triangulation $F' = F \cup G(\alpha)$ for $C(G)$
and vice versa.  Further, the chain orderings correspond to reversed
prefixes of PEOs and vice versa.  We show: There exists a chain
completion of cost at most $k$ if and only if we can triangulate
$C(G)$ with FLOP count of at most $k' \defby k + p(p+1)^2 +
\sum_{i=1}^{p+q} i^2$.

If $F$ is a set of edges whose addition to $G$ yields a chain graph
$G^+$ with chain ordering $\alpha$ for $P$, then $\alpha^R$ starts a PEO for
the corresponding triangulation of $C(G)$.  We will calculate the
elimination degrees.  At the $i$-th elimination step, the vertex
$\alpha^R(i)$ is adjacent to $p-i$ vertices in $P$,
$\vdeg[G^+]{\alpha^R(i)}$ vertices in $Q$ and $i$ vertices in $U$.  So
the $p$ elimination degrees associated with $\alpha^R$ are
\begin{equation}\label{eqn:elim_degree_qcc}
\begin{split}
    \vdeg{\alpha^R(i)} & = p - i + \vdeg[G^+]{\alpha^R(i)} + i\\
    & = p + \vdeg[G^+]{\alpha^R(i)}, \quad 1 \le i \le p.
\end{split}
\end{equation}
After the elimination of these first $p$ vertices, a clique of size
$p+q$ remains, so a PEO $\alpha'$ for $C^+$ is obtained by completing
$\alpha^R$ arbitrarily.  For the FLOP count we find:
\begin{align*}
    \flopop{\alpha'} & = 
        \sum_{i=1}^p (p + 1 + \vdeg[G^+]{\alpha^R(i)})^2
        + \sum_{i=1}^{p+q} i^2\\
    & = \sum_{v \in P} \vdeg[G^+]{v}^2
        + 2(p+1) \sum_{v \in P} \vdeg[G^+]{v}
        + p(p+1)^2 + \sum_{i=1}^{p+q} i^2\\
    & = \qcc(F) + p(p+1)^2 + \sum_{i=1}^{p+q} i^2.
\end{align*}

Since the FLOP count does not depend on the particular PEO $\alpha'$
for $C^+$, the FLOP count induced by the triangulation $F'$ is less
than $k'$ if and only if the quadratic chain completion cost of $F$ is
less than $k$.
\end{proof}

If we would omit the vertices $U$ from the construction of $C(G)$,
the vertex degrees \eqref{eqn:elim_degree_qcc} would depend on the
position of the vertices in the ordering $\alpha$.  The implied
quadratic cost function for the chain completion problem would make
the treatment that follows much more difficult.

\subsubsection{Reduction from optimal quadratic arrangement}
\label{sec:reduction_from_oqa}

In section \ref{sec:OQA} we have shown that OQA($c$) is an NP hard
problem for any choice of the polynomial $c$ in
\eqref{eqn:c_polynomial}.  For the rest of the section we are
interested only in the special case OQA($2(X^2 +1)$), which we
reduce to the QCC problem.  This polynomial is intentionally chosen to
match up with the $2(p+1)$ factor in the formulation of the QCC
problem.

\begin{figure}[t]
    \begin{center}
        \includegraphics{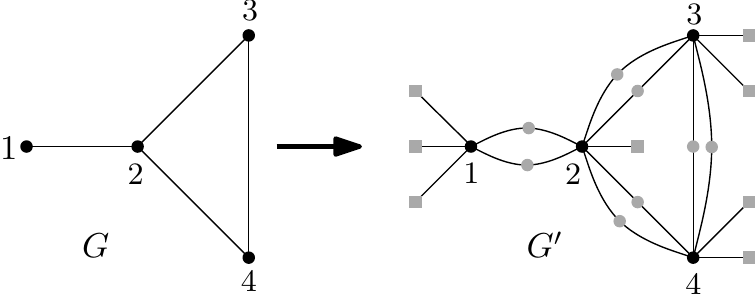}
    \end{center}
    \caption{Illustration for the reduction from OQA to QCC.  The gray
    vertices of $G'$ form the partition $Q$.  The vertices marked as
    squares correspond to the sets $R(v)$, and the gray discs correspond to the
    vertices $w_i^e$.\label{fig:oqa_to_qcc}}
\end{figure}

The following construction for creating a bipartite graph $G' = (P, Q,
E')$ from a given graph $G=(V,E)$ on $n$ vertices is used
in~\cite[Lemma 3]{Yannakakis:1981}.  For a vertex $v \in V$ define
the set $R(v) \defby \{w^v_1, \dotsc, w^v_{l_v} \setcond l_v = n -
\vdeg[G]{v} \}$, then $G'$ is given by (see Fig.~\ref{fig:oqa_to_qcc})
\begin{equation}
\label{eqn:def_g_prime}
\begin{split}
P  & = V,\quad
Q    = \{w^e_1, w^e_2 \setcond e \in E \} \bigcup_{v\in V} R(v)\quad\text{and}\\
E' & = \{ (u, w^e_i) \setcond e \in E, u \in V, e \in \ecut{u},
    1 \le i \le 2 \}\\
& \quad \cup \{ (v, w) \setcond v \in V, w \in R(v)\}.
\end{split}
\end{equation}

The construction of $G'$ is such that all inclusion minimal chain
completions can be easily characterized from vertex orderings of
$G$, as the next lemma shows.
\begin{lemma}[extracted from {\cite[Lemma 3]{Yannakakis:1981}}]
\label{lemma:orderings_chaincompletions}
Let $\alpha : V \rightarrow \nset{n}$ be an ordering for the vertices of $G
= (V,E)$ and for $w \in Q$, define $\sigma(w) = \max \{i \setcond (w,
\alpha^{-1}(i)) \in E' \}$. Then
\begin{equation}
H(\alpha) = \{(\alpha^{-1}(j), w) \setcond w \in Q, j < \sigma(w) \}
    \setminus E' \subseteq P \times Q
\end{equation}
is a set of edges whose addition to $G'$ yields a bipartite chain
graph with chain ordering $\alpha$ for $P$.  Moreover, for any minimal
set of edges $F$ such that $(P,Q,E'\cup F)$ is a bipartite chain graph
with $P$-ordering $\alpha$, we have $F = H(\alpha)$.
\end{lemma}

\begin{theorem}
Let $c = 2(X^2 + 1)$, then OQA(c) $\propto$ QCC.
\end{theorem}
\begin{proof}
Let $(G = (V,E); k)$ be an instance of OQA with $\card{V} = n, \card{E} =
m$.  Let $G'$ be constructed as in \eqref{eqn:def_g_prime}.  We define
an instance for QCC by $(G'; k+p(n))$, where $p(n) = \frac{1}{6} n^2
(n+1) (2n+3c(n)+1)$, and regard $Q$ as the designated partition for the
decision problem.  For the number of vertices in $Q$ we find
\begin{equation*}
    \card{Q} = 2m + \sum_{v \in V} \card{R(v)}
        = 2m + \sum_{v \in V} n - \vdeg[G]{v}
        = 2m + n^2 - 2m = n^2.
\end{equation*}
By Lemma \ref{lemma:orderings_chaincompletions}, we only need to
relate the quadratic ordering cost of an arbitrary vertex ordering
$\alpha : V \rightarrow \nset{n}$ for $G$ to the quadratic chain completion
cost for $H(\alpha)$ for $G'$.  Set $G'^+ = (P,Q,E' \cup H(\alpha))$
and assume for all edges $e = (u,v) \in E$ that we have
$\alpha(u) < \alpha(v)$.  For every vertex $w^e_i \in Q$, we have
$\vdeg[G'^+]{w^e_i} = \alpha(v)$.  For any $v \in V$ we have
$\vdeg[G'^+]{w} = \alpha(v)$ for all vertices $w \in R(v)$.
We abbreviate $l_v \defby n - \vdeg[G]{v}$ and find for the total
quadratic chain completion cost:
\begin{align*}
    & \qcc(H(\alpha))
        = \sum_{w \in Q} (\vdeg[G'^+]{w}^2
        + \underbrace{2(n^2 + 1)}_{=c(n)} \vdeg[G'^+]{w})\\
    & = 2 \sum_{(u,v) \in E} (\alpha(v)^2 + c(n) \alpha(v))
        + \sum_{v \in V} \sum_{x \in R(v)} (\alpha(v)^2 + c(n) \alpha(v))
        \displaybreak[1]\\
    & = 2 \sum_{(u,v) \in E} (\alpha(v)^2 + c(n) \alpha(v))
        + \sum_{v \in V} (n-\vdeg[G]{v})(\alpha(v)^2 + c(n) \alpha(v))\\
    & \quad + \sum_{(u,v) \in E} (\alpha(u)^2 + c(n) \alpha(u))
        - \sum_{(u,v) \in E} (\alpha(u)^2 + c(n) \alpha(u))
        \displaybreak[1]\\
    & = \sum_{(u,v) \in E}
        (\alpha(v)^2 - \alpha(u)^2 + c(n) (\alpha(v) - \alpha(u)))\\
    & \quad + \sum_{(u,v) \in E}
            (\alpha(v)^2 + \alpha(u)^2 + c(n)(\alpha(u) + \alpha(v))
        + \sum_{v \in V} l_v(\alpha(v)^2 + c(n) \alpha(v))
        \displaybreak[1]\\
    & = \quadcost(\alpha) + \sum_{v \in V} \vdeg[G]{v}
        (\alpha(v)^2 + c(n) \alpha(v))
        + \sum_{v \in V} (n-\vdeg[G]{v})(\alpha(v)^2 + c(n) \alpha(v))
        \displaybreak[1]\\
    & =  \quadcost(\alpha) + n \sum_{v \in V} (\alpha(v)^2 + c(n) \alpha(v))
        = \quadcost(\alpha) + p(n).
\end{align*}
This shows $\quadcost(\alpha) \le k \Leftrightarrow \qcc(H(\alpha))
\le k + p(n)$, which completes the proof.
\end{proof}

Taken together, the three reductions in this section imply that it is
NP hard to minimize the number of arithmetic operations in the
sparse Cholesky factorization.

\section{Conclusions and future work}

In this work we have shown by means of an explicit, scalable
construction that minimum fill and minimum operation count for the
sparse Cholesky factorization are not achievable simultaneously in
general.  We proved that minimizing the number of arithmetic
operations is just as difficult as minimizing the fill: it is NP
hard.  While this result is not surprising, no proof has been given so
far, and thus our findings close a gap in the theoretical body of
sparse direct methods.

It would be of interest to understand how well optimal fill orderings
\emph{approximate} the optimal number of arithmetic operations (and
vice versa).  Approximation bounds based on general equivalence
constants for the 1- and 2-norm or bounds based on full $k$-tree
embeddings (e.g.  \cite[prop.  3]{Rose:1972}) are too coarse for
offering an quantitative insight into this question.

\paragraph*{Acknowledgements} We would like to thank Tim Davis and
John Gilbert for inspiring discussions at a Dagstuhl workshop which
encouraged this work.  We are also grateful to two anonymous
referees for their useful suggestions.

Research at Lawrence Berkeley National Laboratory was
supported by the Office of Advanced Scientific Computing Research of
the US Department of Energy under contract number DE-AC02-05CH11231.

\bibliographystyle{plain}
\bibliography{fillflop}

\begin{thebibliography}{10}

\bibitem{AgrawalKleinRavi:1993}
Ajit Agrawal, Philip Klein, and R.~Ravi.
\newblock Cutting down on fill using nested dissection: provably good
  elimination orderings.
\newblock In {\em Graph theory and sparse matrix computation}, volume~56 of
  {\em IMA Vol. Math. Appl.}, pages 31--55. Springer, New York, 1993.

\bibitem{Amestoy:1996}
Patrick~R. Amestoy, Timothy~A. Davis, and Iain~S. Duff.
\newblock An approximate minimum degree ordering algorithm.
\newblock {\em SIAM J. Matrix Anal. Appl.}, 17(4):886--905, 1996.

\bibitem{AmestoyDavisDuff:2004}
Patrick~R. Amestoy, Timothy~A. Davis, and Iain~S. Duff.
\newblock Algorithm 837: {AMD}, an approximate minimum degree ordering
  algorithm.
\newblock {\em ACM Trans. Math. Software}, 30(3):381--388, 2004.

\bibitem{ArnborgEtal:1987}
Stefan Arnborg, Derek~G. Corneil, and Andrzej Proskurowski.
\newblock Complexity of finding embeddings in a {$k$}-tree.
\newblock {\em SIAM J. Algebraic Discrete Methods}, 8(2):277--284, 1987.

\bibitem{BodlaenderEtal:1995}
Hans~L. Bodlaender, John~R. Gilbert, Hj{\'a}lmt{\'y}r Hafsteinsson, and Ton
  Kloks.
\newblock Approximating treewidth, pathwidth, frontsize, and shortest
  elimination tree.
\newblock {\em J. Algorithms}, 18(2):238--255, 1995.

\bibitem{DavisHu:2011}
Timothy~A. Davis and Yifan Hu.
\newblock The {U}niversity of {F}lorida sparse matrix collection.
\newblock {\em ACM Trans. Math. Software}, 38(1):Art. 1, 25, 2011.

\bibitem{DuffReid:1983:mf}
I.~S. Duff and J.~K. Reid.
\newblock The multifrontal solution of indefinite sparse symmetric linear
  equations.
\newblock {\em ACM Trans. Math. Software}, 9(3):302--325, 1983.

\bibitem{DuffReid:1983:work}
I.~S. Duff and J.~K. Reid.
\newblock A note on the work involved in no-fill sparse matrix factorization.
\newblock {\em IMA J. Numer. Anal.}, 3(1):37--40, 1983.

\bibitem{Even:1979}
Shimon Even.
\newblock {\em Graph algorithms}.
\newblock Computer Science Press Inc., Woodland Hills, Calif., 1979.
\newblock Computer Software Engineering Series.

\bibitem{GareyJohnson:1979}
Michael~R. Garey and David~S. Johnson.
\newblock {\em Computers and intractability}.
\newblock W. H. Freeman and Co., San Francisco, Calif., 1979.
\newblock A guide to the theory of NP-completeness, A Series of Books in the
  Mathematical Sciences.

\bibitem{GeorgeLiu:1989}
Alan George and Joseph W.~H. Liu.
\newblock The evolution of the minimum degree ordering algorithm.
\newblock {\em SIAM Rev.}, 31(1):1--19, 1989.

\bibitem{GeorgePothen:1997}
Alan George and Alex Pothen.
\newblock An analysis of spectral envelope reduction via quadratic assignment
  problems.
\newblock {\em SIAM J. Matrix Anal. Appl.}, 18(3):706--732, 1997.

\bibitem{Heggernes:2006}
Pinar Heggernes.
\newblock Minimal triangulations of graphs: a survey.
\newblock {\em Discrete Math.}, 306(3):297--317, 2006.

\bibitem{KarypisKumar:1998}
George Karypis and Vipin Kumar.
\newblock A fast and high quality multilevel scheme for partitioning irregular
  graphs.
\newblock {\em SIAM J. Sci. Comput.}, 20(1):359--392 (electronic), 1998.

\bibitem{Kloks:1994}
Ton Kloks.
\newblock {\em Treewidth}, volume 842 of {\em Lecture Notes in Computer
  Science}.
\newblock Springer-Verlag, Berlin, 1994.
\newblock Computations and approximations.

\bibitem{Liu:1985}
Joseph W.~H. Liu.
\newblock Modification of the minimum-degree algorithm by multiple elimination.
\newblock {\em ACM Trans. Math. Software}, 11(2):141--153, 1985.

\bibitem{Liu:1992}
Joseph W.~H. Liu.
\newblock The multifrontal method for sparse matrix solution: theory and
  practice.
\newblock {\em SIAM Rev.}, 34(1):82--109, 1992.

\bibitem{NatanzonEtal:2000}
Assaf Natanzon, Ron Shamir, and Roded Sharan.
\newblock A polynomial approximation algorithm for the minimum fill-in problem.
\newblock {\em SIAM J. Comput.}, 30(4):1067--1079 (electronic), 2000.

\bibitem{NgRaghavan:1999}
Esmond~G. Ng and Padma Raghavan.
\newblock Performance of greedy ordering heuristics for sparse {C}holesky
  factorization.
\newblock {\em SIAM J. Matrix Anal. Appl.}, 20(4):902--914 (electronic), 1999.
\newblock Sparse and structured matrices and their applications (Coeur d'Alene,
  ID, 1996).

\bibitem{Papadimitriou:1976}
Ch.~H. Papadimitriou.
\newblock The {NP}-completeness of the bandwidth minimization problem.
\newblock {\em Computing}, 16(3):263--270, 1976.

\bibitem{MehtaSahni:2005}
Alex Pothen and Sivan Toledo.
\newblock {\em Handbook of data structures and applications}, chapter~59.
\newblock Chapman \& Hall/CRC Computer and Information Science Series. Chapman
  \& Hall/CRC, Boca Raton, FL, 2005.

\bibitem{Rose:1972}
Donald~J. Rose.
\newblock A graph-theoretic study of the numerical solution of sparse positive
  definite systems of linear equations.
\newblock In {\em Graph theory and computing}, pages 183--217. Academic Press,
  New York, 1972.

\bibitem{RothbergEisenstat:1998}
Edward Rothberg and Stanley~C. Eisenstat.
\newblock Node selection strategies for bottom-up sparse matrix ordering.
\newblock {\em SIAM J. Matrix Anal. Appl.}, 19(3):682--695 (electronic), 1998.

\bibitem{Yannakakis:1981}
Mihalis Yannakakis.
\newblock Computing the minimum fill-in is {NP}-complete.
\newblock {\em SIAM J. Algebraic Discrete Methods}, 2(1):77--79, 1981.

\end{thebibliography}

\opt{arxiv}{\newpage

\begin{appendix}

\section{OQA(c) and OLA are different}
\label{app:oqa_vs_ola}

We show that OLA and OQA are different problems in the sense that
optimizing the linear arrangement cost does not necessarily optimize
the quadratic arrangement cost (and vice versa).  Let $n>4$, $C$ a set
of size $n$, $u,v \in C$ two distinct elements and consider the
following class of graphs $G= (V,E)$ (see Fig.
\ref{fig:two_antennas}):
\begin{equation*}
    V = C \cup \{x,y\}, \quad
    E = C\times C \cup \{(x, u), (y, v)\}
\end{equation*}

\begin{figure}
    \begin{center}
        \includegraphics[width=.3\textwidth]{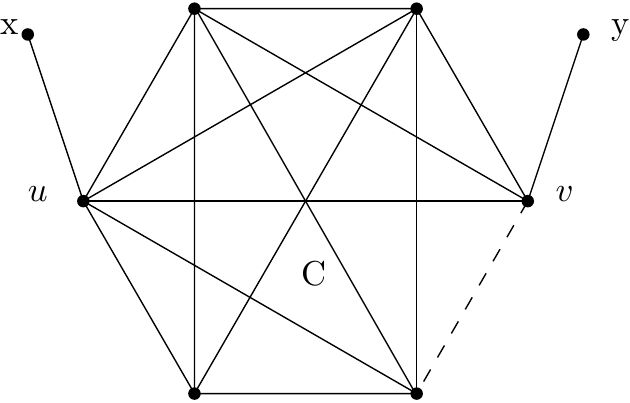}
        \hfill
        \includegraphics[width=.6\textwidth]{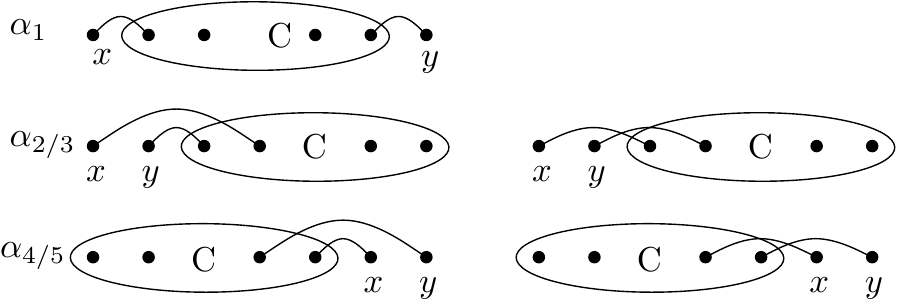}
        \caption{A graph where minimum quadratic- and linear
        arrangement costs are attained on distinct
        orderings.\label{fig:two_antennas}}
    \end{center}
\end{figure}

It is easy to see that any linear or quadratic arrangement where $x$
or $y$ intersperses with the vertices of $C$ is suboptimal.  If $x,y$
are ordered before or after $C$, any ordering that does not place $u,
v$ as close as possible to $x,y$ is suboptimal, too.  Ruling out those
suboptimal orderings, only five different orderings (modulo
cost-neutral rearrangements of $C$) remain; they are displayed in Fig.
\ref{fig:two_antennas} on the right.  We calculate the linear
arrangement costs:
\begin{equation*}
    \lincost(\alpha_1)     = \frac{1}{6} n(n^2 -1) + 2,\quad
    \lincost(\alpha_{2/3}) = \lincost(\alpha_{4/5}) =\frac{1}{6} n(n^2 -1) + 4,
\end{equation*}
so $\alpha_1$ is an optimal linear arrangement while the others are
not.  Using Lemma \ref{lemma:quadcost_clique_trans} we find the
quadratic arrangement costs
\begin{align*}
    \quadcost(\alpha_1)    & = \frac{1}{6} n(n^2-1)(c + n+3) + 2n + c + 3,\\
    \quadcost(\alpha_{2/3})& = \frac{1}{6} n(n^2-1)(c + n+5) + 4c + 20,\\
    \quadcost(\alpha_{4/5})& = \frac{1}{6} n(n^2-1)(c + n+1) + 8n + 4c + 4.
\end{align*}
It is easy to see that $\quadcost(\alpha_{4/5})$ is strictly less than
the other costs for sufficiently large $n$ (recall that $c$ is
fixed).  Thus OQA and OLA are different problems for every
polynomial $c$.

\newpage

\section{Moving isolated vertices to the left}
\label{app:ex_move_left}

In the reduction from \probname{MaxCut} to OQA(c), we needed to
rearrange isolated vertices within a given ordering without decreasing
the costs.  Fig. \ref{fig:no_single_left} shows an arrangement of a
graph on 8 vertices, of which one is isolated.  If the isolated vertex
is moved to the left so that it intersperses with the largest cut, 
the arrangement cost decreases.  This is why we need to resort to
rearranging blocks of isolated vertices.

\begin{figure}
    \begin{center}
        \includegraphics[width=.6\textwidth]{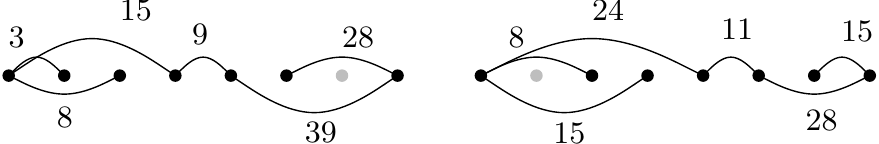}
\caption{Moving an isolated vertex to the left into the largest cut
may decrease the total arrangement cost.  The numbers shown next to
the edges are the edge costs; the arrangement on the left has a total
cost of 102 while the arrangement on the right has cost 101.  Here the
distance function is $f(x) = x^2$.
\label{fig:no_single_left}}
    \end{center}
\end{figure}

\newpage

\section{Auxiliary proofs}
\label{app:proof_quadcost_clique_trans}

\paragraph{Proof of Lemma \ref{lemma:quadcost_clique_trans}}

We shall first proof the lemma with $r=0$.
We say that the vertices $u,v$ have ordering distance $d$,  $1 \le d
\le s-1$, if $|\alpha(u) - \alpha(v)| = d$.  The cost implied by all
edges between vertices of ordering distance $d$ is
\begin{equation*}
\begin{split}
    & \quad \sum_{k=1}^{s-d} \left( (k+d)^2 + c(k+d) - k^2 - ck \right)
      = 2d\sum_{k=1}^{s-d} k + d(d + c) (s-d)\\
    & = d(s - d)(s - d + 1) + d(d + c)(s - d)
      = d (c + s + 1)(s - d).
\end{split}
\end{equation*}
The total cost of $\alpha$ is the sum
over all the distances $1\le d\le
s-1$, so we find
\begin{equation*}
    \begin{split}
    \quadcost(\alpha)
        & = \sum_{d=1}^{s-1} d(c+s+1)(s-d)
        = (c+s+1) (s \sum_{d=1}^{s-1} d - \sum_{d=1}^{s-1} d^2)\\
        & = (c+s+1)(\frac{1}{2} s^2(s-1) - \frac{1}{6} s (s-1) (2s -1))\\
        & = \frac{1}{6} s (s^2-1) (c + s + 1).
\end{split}
\end{equation*}

Application of the translation lemma to the above cost finishes
the proof.
\qed

\end{appendix}

}

\end{document}